\documentclass[12pt]{amsart}

\usepackage{microtype}

\usepackage{times}
\usepackage{pifont}
\usepackage{datetime}
\usepackage[normalem]{ulem}

\usepackage{amsmath,amssymb,commath,mathtools,mathabx,bbm}
\usepackage{xspace}

\usepackage[T2A,T1]{fontenc}
\usepackage[utf8]{inputenc}
\usepackage[russian,english]{babel}

\usepackage{dsfont}

\usepackage{hyperref}
\usepackage[capitalize,noabbrev]{cleveref}
\usepackage{enumitem}
\usepackage{units}
\usepackage{subfiles}
\usepackage[labelfont={normalsize}]{caption,subfig}

\usepackage[backgroundcolor=yellow]{todonotes}

\newcommand{\Sniady}{the last named author\xspace}

\usepackage{tikz}
\usetikzlibrary{calc}

\DeclareUnicodeCharacter{2212}{-}
\DeclareUnicodeCharacter{10FC0B}{*}
\DeclareUnicodeCharacter{10FC07}{*}
\DeclareUnicodeCharacter{10FC08}{*}
\DeclareUnicodeCharacter{10FC14}{*}
\DeclareUnicodeCharacter{0302}{*}
\DeclareUnicodeCharacter{0308}{*}

\usepackage[backend=bibtex,doi=true,isbn=false,url=true,style=alphabetic]{biblatex}
\newcommand{\biblio}{\printbibliography}

\renewbibmacro*{doi+eprint+url}{%
    \printfield{doi}%
    \newunit\newblock%
    \iftoggle{bbx:eprint}{%
        \usebibmacro{eprint}%
    }{}%
    \newunit\newblock%
    \iffieldundef{doi}{%
        \usebibmacro{url+urldate}}%
    {}%
}

\makeatletter
\@ifclasswith{amsart}{externalize}%
{
	\usetikzlibrary{external} 
	\tikzexternalize[
	optimize=false,
	prefix=tikz/,
	mode=list and make]
	\makeatletter
	\renewcommand{\todo}[2][]{\tikzexternaldisable\@todo[#1]{#2}\tikzexternalenable}
	\makeatother
}%
{
	\newcommand{\tikzexternaldisable}{}
	\newcommand{\tikzexternalenable}{}
}
\makeatother

\author[M.~Marciniak]{Mikołaj Marciniak} 

\address{Interdisciplinary Doctoral School “Academia Copernicana”, Faculty of
    Mathematics and Computer Science, Nicolaus Copernicus University in Toruń,
    ul.~Chopina 12/18, 87-100 Toruń, Poland} 
\email{marciniak@mat.umk.pl}

\author[{\L}.~Maślanka]{{\L}ukasz Ma\'slanka}

\address{Institute of Mathematics, Polish Academy of Sciences,
    \mbox{ul.~\'Sniadeckich 8}, 00-656 Warszawa, Poland}

\email{lmaslanka@impan.pl}

\author[P.~\'Sniady]{Piotr \'Sniady}

\address{Institute of Mathematics, Polish Academy of Sciences,
    \mbox{ul.~\'Sniadec\-kich 8,} \linebreak 00-656 Warszawa, Poland }

\email{psniady@impan.pl}

\theoremstyle{definition}
\newtheorem{definition}{Definition}[section]

\theoremstyle{plain}
\newtheorem{theorem}[definition]{Theorem}
\newtheorem{proposition}[definition]{Proposition}
\newtheorem{corollary}[definition]{Corollary}
\newtheorem{lemma}[definition]{Lemma}

\newtheorem{conjecture}[definition]{Conjecture}

\theoremstyle{remark}

\newtheorem{remark}[definition]{Remark}
\newtheorem{question}[definition]{Question}
\newtheorem{problem}[definition]{Problem}

\newcommand{\N}{\mathbb{N}}
\newcommand{\R}{\mathbb{R}}

\newcommand{\PP}{\mathbb{P}}
\newcommand{\PPcond}[2]{\PP\left(#1  \;\middle\vert\; #2 \right)}
\newcommand{\PPcondCurly}[2]{\PP\left\{ #1  \;\middle\vert\; #2 \right\}}

\newcommand{\E}{\mathbb{E}}
\newcommand{\Z}{\mathbb{Z}}
\newcommand{\n}{^{[n]}}

\newcommand{\ind}{\mathds{1}}

\newcommand{\Sym}[1]{\mathfrak{S}_#1}
\newcommand{\tableau}{\mathcal{T}}
\newcommand{\diagrams}{\mathbb{Y}}

\newcommand{\myPset}{\mathcal{P}^{(n)}}

\DeclareMathOperator{\RSK}{RSK}
\newcommand{\inrv}{\overset{\operatorname{rv}}{\in}}

\DeclareMathOperator{\Plancherel}{Plan}

\newcommand{\AD}{\cite[Theorem~5(b)]{Aldous1995}\xspace}

\begin{document}

\title[Poisson limit theorems for Robinson--Schensted correspondence]%
{Poisson limit theorems \\ for the Robinson--Schensted correspondence \\
and for the multi-line Hammersley process}

\begin{abstract}
We consider the Robinson--Schensted--Knuth algorithm applied to a random input and
study the growth of the bottom rows of the corresponding Young diagrams. We
prove a multidimensional Poisson limit theorem for the resulting Plancherel
growth process. In this way we extend the result of Aldous and Diaconis to more
than just one row. This result can be interpreted as convergence of the
multi-line Hammersley process to its stationary distribution which is given by
a collection of independent Poisson point processes.
\end{abstract}



\maketitle

\section{Introduction}
\label{sec:preliminaries}

\subsection{Notations}

The set of Young diagrams will be denoted by $\diagrams$; the set of Young
diagrams with $n$ boxes will be denoted by $\diagrams_n$. The set $\diagrams$
has a~structure of an oriented graph, called \emph{the Young graph}; a pair
$\mu\nearrow\lambda$ forms an oriented edge in this graph if the Young diagram
$\lambda$ can be created from the Young diagram $\mu$ by addition of a single
box.

We will draw Young diagrams and tableaux in the French convention with the
Cartesian coordinate system $Oxy$, see~\cref{fig:RSKa}. 
We
index the rows and the columns of tableaux by \emph{non-negative} integers from
$\N_0=\{0,1,2,\dots\}$. In particular, if $\Box$ is a box of a tableau, we
identify it with the Cartesian coordinates of its \emph{lower-left corner}:
$\Box=(x,y)\in \N_0\times \N_0$. For a tableau $\tableau$ we denote by
$\tableau_{x,y}$ its entry which lies in the intersection of the row $y\in\N_0$
and the column $x\in\N_0$.

Also the rows of any Young diagram $\lambda=(\lambda_0,\lambda_1,\dots)$ are
indexed by the~elements of~$\N_0$; in particular the length of the bottom row
of $\lambda$ is denoted by~$\lambda_0$.

\medskip

Let  $X\colon \Omega \to V$ be a random variable with values in some set $V$. When we want to phrase this statement without mentioning the sample space $\Omega$
explicitly, we will write $X\inrv V$.

If $E$ is a random event, we denote by $\ind_E$ its indicator, which is the random variable given by
\[ \ind_E (\omega) = \begin{cases} 1 & \text{if $\omega\in E$}, \\
                                    0 & \text{otherwise.} 
                                \end{cases}
                                    \]

\subsection{Schensted row insertion}

\begin{figure}[t]
    \centering
    \hfill
    \subfloat[]{
        \begin{tikzpicture}[scale=0.75]
        \clip (-0.1,-0.5) rectangle (4.1,4.5);

        \draw (0,0) rectangle +(1,1); 
        \node at (0.5,0.5) {16}; 
        \draw (1,0) rectangle +(1,1); 
        \node at (1.5,0.5) {37}; 
        \draw (2,0) rectangle +(1,1); 
        \node at (2.5,0.5) {41}; 
        \draw (3,0) rectangle +(1,1); 
        \node at (3.5,0.5) {82}; 
        \draw (0,1) rectangle +(1,1); 
        \node at (0.5,1.5) {23}; 
        \draw (1,1) rectangle +(1,1); 
        \node at (1.5,1.5) {53}; 
        \draw (2,1) rectangle +(1,1); 
        \node at (2.5,1.5) {70}; 
        \draw (0,2) rectangle +(1,1); 
        \node at (0.5,2.5) {74}; 
        \draw (1,2) rectangle +(1,1); 
        \node at (1.5,2.5) {99};

        \end{tikzpicture}
        \label{fig:RSKa}
    }
    \hfill
    \subfloat[]{
        \begin{tikzpicture}[scale=0.75]
        \clip (-1.5,-0.5) rectangle (4.1,4.5);
        \fill[blue!10] (1,0) rectangle +(1,1);
        \fill[blue!10] (1,1) rectangle +(1,1);
        \fill[blue!10] (0,2) rectangle +(1,1);
        \fill[blue!10] (0,3) rectangle +(1,1);

        \draw (0,0) rectangle +(1,1); 
        \node at (0.5,0.5) {16}; 
        \draw (1,0) rectangle +(1,1); 
        \node at (1.5,0.5) {37}; 
        \draw (2,0) rectangle +(1,1); 
        \node at (2.5,0.5) {41}; 
        \draw (3,0) rectangle +(1,1); 
        \node at (3.5,0.5) {82}; 
        \draw (0,1) rectangle +(1,1); 
        \node at (0.5,1.5) {23}; 
        \draw (1,1) rectangle +(1,1); 
        \node at (1.5,1.5) {53}; 
        \draw (2,1) rectangle +(1,1); 
        \node at (2.5,1.5) {70}; 
        \draw (0,2) rectangle +(1,1); 
        \node at (0.5,2.5) {74}; 
        \draw (1,2) rectangle +(1,1); 
        \node at (1.5,2.5) {99};

        \draw[->] (-0.3,-0.45) to[bend left=60] (-0.3,0.45);
        \draw[->] (0,1) +(-0.3,-0.45) to[bend left=60] +(-0.3,0.45);
        \draw[->] (0,2) +(-0.3,-0.45) to[bend left=60] +(-0.3,0.45);
        \draw[->] (0,3) +(-0.3,-0.45) to[bend left=60] +(-0.3,0.45);
        
        \tiny
        \node[] at (-0.85,0) {18};
        \node[] at (-0.85,1) {37};
        \node[] at (-0.85,2) {53};
        \node[] at (-0.85,3) {74};
        
        \end{tikzpicture}
        \label{fig:RSKb}
    }
    \hfill
    \subfloat[]{
        \begin{tikzpicture}[scale=0.75]
        \clip (-1.5,-0.5) rectangle (4.1,4.5);
        \fill[blue!10] (1,0) rectangle +(1,1);
        \fill[blue!10] (1,1) rectangle +(1,1);
        \fill[blue!10] (0,2) rectangle +(1,1);
        \fill[blue!10] (0,3) rectangle +(1,1);
        
        \draw (0,0) rectangle +(1,1); 
        \node at (0.5,0.5) {16}; 
        \draw (1,0) rectangle +(1,1); 
        \node at (1.5,0.5) {18}; 
        \draw (2,0) rectangle +(1,1); 
        \node at (2.5,0.5) {41}; 
        \draw (3,0) rectangle +(1,1); 
        \node at (3.5,0.5) {82}; 
        \draw (0,1) rectangle +(1,1); 
        \node at (0.5,1.5) {23}; 
        \draw (1,1) rectangle +(1,1); 
        \node at (1.5,1.5) {37}; 
        \draw (2,1) rectangle +(1,1); 
        \node at (2.5,1.5) {70}; 
        \draw (0,2) rectangle +(1,1); 
        \node at (0.5,2.5) {53}; 
        \draw (1,2) rectangle +(1,1); 
        \node at (1.5,2.5) {99}; 
        \draw (0,3) rectangle +(1,1); 
        \node at (0.5,3.5) {74};

        \draw[->] (-0.3,-0.45) to[bend left=60] (-0.3,0.45);
        \draw[->] (0,1) +(-0.3,-0.45) to[bend left=60] +(-0.3,0.45);
        \draw[->] (0,2) +(-0.3,-0.45) to[bend left=60] +(-0.3,0.45);
        \draw[->] (0,3) +(-0.3,-0.45) to[bend left=60] +(-0.3,0.45);
        
        \tiny
        \node[] at (-0.85,0) {18};
        \node[] at (-0.85,1) {37};
        \node[] at (-0.85,2) {53};
        \node[] at (-0.85,3) {74};
        
        \end{tikzpicture}
        \label{fig:RSKc}
    }
    
    \caption{\protect\subref{fig:RSKa} The original tableau $\tableau$.
    \protect\subref{fig:RSKb} We consider the Schensted row insertion of the number
    $18$ to the~tableau~$\tableau$. The~highlighted boxes form the corresponding
    bumping route. The small numbers on the left (next
    to the arrows) indicate the~inserted/bumped numbers. \protect\subref{fig:RSKc} The
    output $\tableau\leftarrow 18$ of the Schensted insertion. } \label{fig:RSK}
\end{figure}

\newcommand{\letter}{a}

\emph{The Schensted row insertion} is an algorithm which takes as an input a
tableau $\tableau$ and some number $\letter$. The number $\letter$ is inserted
into the first row (that is,~the bottom row, the row with the index~$0$) of
$\tableau$ to the~leftmost box which contains an entry which is strictly bigger
than $\letter$.

In the case when the row contains no entries which
are bigger than~$\letter$, the~number~$\letter$ is inserted into the leftmost
empty box in this row and the algorithm terminates. 

If, however, the number $\letter$ is inserted into a box which was not empty,
the~previous content $\letter'$ of the box is \emph{bumped} into the second row
(that is, the~row with the index $1$). This means that the algorithm is
iterated but this time the number $\letter'$ is inserted into the second row to
the leftmost box which contains a number bigger than $\letter'$. 
We repeat these steps of row insertion and bumping 
until some number is inserted into a previously empty box.
This process is illustrated on \cref{fig:RSKb,fig:RSKc}. The outcome of
the Schensted insertion is defined as the~result of the~aforementioned procedure;
it will be denoted by $\tableau \leftarrow \letter$.

\subsection{Robinson--Schensted--Knuth algorithm}

For the purposes of this article we consider a simplified version of \emph{the
    Robinson--Schensted--Knuth algorithm}; for this reason we should rather call it
\emph{the Robinson--Schensted algorithm}. Nevertheless, we use the first name
because of its well-known acronym RSK. The RSK algorithm associates to a finite
sequence $w=(w_1,\dots,w_\ell)$ a~pair of tableaux: \emph{the insertion tableau
    $P(w)$} and \emph{the recording tableau~$Q(w)$}.

The insertion tableau 
\begin{equation}
\label{eq:insertion}
P(w) = \Big( \big( (\emptyset \leftarrow w_1) \leftarrow w_2 \big) 
                                      \leftarrow  \cdots \Big) \leftarrow w_\ell
\end{equation}
is defined as the result of the iterative Schensted insertion applied to the
entries of the sequence $w$, starting from \emph{the empty tableau $\emptyset$}.

The recording tableau $Q(w)$ is defined as the standard Young tableau of the
same shape as $P(w)$ in which each entry is equal to the number of 
the~iteration of \eqref{eq:insertion} in which the given box of~$P(w)$ stopped
being empty; in other words the entries of $Q(w)$ give the order in which the
entries of the~insertion tableau were filled.

The tableaux $P(w)$ and $Q(w)$ have the same shape; we will denote this common
shape by $\RSK(w)$ and call it \emph{the RSK shape associated to $w$}.

The RSK algorithm is of great importance in algebraic combinatorics, especially in
the context of the representation theory \cite{Fulton1997}.

\subsection{Plancherel measure, Plancherel growth process}
\label{sec:pgp}

Let $\Sym{n}$ denote the symmetric group of order $n$. We will view each
permutation $\pi \in\Sym{n}$ as a sequence $\pi=(\pi_1,\dots,\pi_n)$ which has
no repeated entries, and such that $\pi_1,\dots,\pi_n\in\{1,\dots,n\}$. A
restriction of RSK to the symmetric group is a bijection which to a given
permutation from $\Sym{n}$ associates a pair $(P,Q)$ of standard Young tableaux
of the same shape, consisting of $n$ boxes. A~fruitful area of study concerns
the RSK algorithm applied to a uniformly random permutation from $\Sym{n}$,
especially asymptotically in the limit $n\to\infty$, see \cite{Romik2015a} and
the references therein.

\emph{The Plancherel measure} on $\diagrams_{n}$, denoted $\Plancherel_n$, is
defined as the probability distribution of the random Young diagram $\RSK(w)$
for a uniformly random permutation $w$ selected from $\Sym{n}$.

\medskip

An \emph{infinite standard Young tableau} 
\cite[Section~2.2]{Kerov1999}
is a filling of the boxes
in a subset of the upper-right quarterplane with positive integers, such that
each row and each column is increasing, and each positive integer is used
exactly once.
There is a natural bijection between the set of infinite standard Young tableaux and the set of infinite sequences of Young diagrams
\begin{equation} 
    \label{eq:PGP0} 
    \lambda^{(0)} \nearrow \lambda^{(1)} \nearrow \cdots \qquad
    \text{with } \lambda^{(0)}=\emptyset;
\end{equation}
this bijection is given by setting $\lambda^{(n)}$ to be the set of boxes of a
given infinite standard Young tableau which are $\leq n$. 

 If
$w=(w_1,w_2,\dots)$ is an \emph{infinite} sequence, the recording tableau $Q(w)$
is defined as the infinite standard Young tableau in which each non-empty entry
is equal to the number of the iteration in the infinite sequence of Schensted insertions
\[   \big( (\emptyset \leftarrow w_1) \leftarrow w_2 \big) 
\leftarrow  \cdots 
\]
in which the corresponding box stopped being empty, see \cite[Section~1.2.4]{RomikSniady-AnnPro}. Under the aforementioned bijection, the recording tableau $Q(w)$ corresponds to the sequence \eqref{eq:PGP0}
with 
\[ \lambda^{(n)}= \RSK( w_1,\dots, w_n). \]

Let $\xi=(\xi_1,\xi_2,\dots)$ be an infinite sequence of independent, identically
distributed random variables with the uniform distribution $U(0,1)$ on the unit
interval $[0,1]$. \emph{The Plan\-che\-rel measure on the set of infinite standard
Young tableaux} is defined as the probability distribution of $Q(\xi)$.
		Any sequence with the same
probability distribution as \eqref{eq:PGP0} with 
\begin{equation}
    \label{eq:lambda-rsk}
    \lambda^{(n)}= \RSK( \xi_1,\dots, \xi_n) 
\end{equation}
will be called \emph{the Plancherel
    growth process} \cite{Kerov1999}. For a~more systematic introduction to this
topic we recommend the monograph \cite[Section~1.19]{Romik2015a}.

\subsection{The main result: Poisson limit theorem for the Plancherel
    growth process}

\newcommand{\magic}{\Delta}

\begin{theorem}
    \label{thm:poisson-growth-Plancherel}
Let $\lambda^{(0)} \nearrow \lambda^{(1)} \nearrow \cdots $ be the Plancherel growth process.
Let us fix $k\in\N_0$.
We denote by
\[ \Lambda^{(n)} = 
\left( \lambda^{(n)}_0, \dots, \lambda^{(n)}_k \right) \inrv \left(\N_0\right)^{k+1}\]
the random vector formed by the lengths of the bottom $k+1$ rows of the random diagram~$\lambda^{(n)}$.
For each $n\in\N_0$ we consider
the random function $\magic_n:\R\to \Z^{k+1}$ given by
\begin{equation} 
    \label{eq:growth-Poisson}
    \magic_n(t)= \Lambda^{(n_t)}-\Lambda^{(n)}
, 
\end{equation}
where
\begin{equation}
    \label{eq:nt}
     n_t = \max\left( n + \lfloor t \sqrt{n} \rfloor,\; 0 \right).
\end{equation}

Then, 
    for $n\to\infty$, the random function $\magic_n$ converges in distribution to a tuple $\left( N_0,\dots,N_k\right)$ of $k+1$ independent copies of the Poisson process with unit intensity.
This convergence is understood as convergence (with respect to the topology given by the total variation distance) of
    all finite-dimensional marginals $\left( \magic_n(t_1) ,\dots, \magic_n(t_\ell) \right)$  over all choices of $t_1,\dots,t_\ell \in \R$.
\end{theorem}

The proof is postponed to \cref{sec:proof-thm:poisson-growth-Plancherel}.

In \cref{sec:cor-1,sec:hammersley-process,sec:multi-line-hammersley-process,sec:link-ulam-hammersley} 
we shall discuss the connections of this theorem with the (multi-line)
Hammersley process and the \emph{assumption $\alpha$} of Hammersley,
and  in
\cref{sec:link-aldous-diaconis} the connections with the work of Aldous and Diaconis \AD. 
In \cref{sec:other-links,sec:generalization} we will
discuss links with some other areas of mathematics.

\begin{remark}
    \label{rem:derivative}
The fact that the intensity of the limiting Poisson process is equal to $1$ can be
justified by the following heuristic argument. 
The first order approximation
for the length of the given fixed row is given by
\[ \E \lambda_i^{(n)} \approx 2 \sqrt{n}, \]
cf.~\eqref{eq:ulam} for the special case of the bottom row.    
By considering the derivative of the right-hand side we can hope that 
\[  \E \lambda_i^{(n+ c \sqrt{n} )} -  \E \lambda_i^{(n)} \approx  c \sqrt{n} \frac{1}{\sqrt{n}} =c\]
which is consistent with the expected growth of the Poisson process with the unit intensity.
\end{remark}

\subsection{Local spacings in the bottom rows of the recording tableau} 
\label{sec:cor-1}

\cref{thm:poisson-growth-Plancherel} can be interpreted as a result about the
random sets of points in which the coordinates of the function
\eqref{eq:growth-Poisson} have jumps; this interpretation gives the following
immediate corollary. For an alternative proof see \cref{sec:proof-localQ}.

\begin{corollary}
\label{coro:localQ}
Let $\xi=(\xi_1,\xi_2,\dots)$ be an infinite sequence of independent, identically
distributed random variables with the uniform distribution $U(0,1)$ on the unit
interval $[0,1]$ and let
$Q(\xi)=\left[ Q_{x,y}\right]_{x,y\geq 0}$ be the corresponding random recording tableau with the Plancherel distribution. 

Then for any integer $k\in\N_0$ 
the collection of $k+1$ random sets
\begin{equation}  
\label{eq:setsQ}
\left( \left\{ \frac{Q_{x,y}-n}{\sqrt{n}} \ : \ x\in \N_0 \right\}  \ \ : \  \  y\in\{0,\dots,k\} \right) 
\end{equation}
converges in distribution, as $n\to\infty$, to a family of $k+1$ independent Poisson point
processes on $\R$ with the unit intensity.
\end{corollary}

\subsection{The Hammersley process}
\label{sec:hammersley-process}

\begin{figure}
    \centering
\subfloat[]{
\begin{tikzpicture}[scale=0.5]
\draw[thick,->] (-1,-1) -- (8,-1) node[anchor=west]{$x$};
\draw[thick,->] (-1,-1) -- (-1,8) node[anchor=south]{$t$};

\draw[blue] (2,0) -- (2,4) -- (1,4) -- (1,8);
\draw[blue] (3,0) -- (3,8);
\draw[blue] (7,0) -- (7,2) -- (5,2) -- (5,6) -- (4,6) -- (4,8);
\draw[blue] (6,5) -- (6,8);

\fill[blue!50!black] (1,4) circle(5pt);
\fill[blue!50!black] (5,2) circle(5pt);
\fill[blue!50!black] (4,6) circle(5pt);
\fill[blue!50!black] (6,5) circle(5pt);

\draw[red,thick] (2,4) +(0.2,0.2)-- +(-0.2,-0.2) +(0.2,-0.2)-- +(-0.2,0.2);
\draw[red,thick] (5,6) +(0.2,0.2)-- +(-0.2,-0.2) +(0.2,-0.2)-- +(-0.2,0.2);
\draw[red,thick] (7,2) +(0.2,0.2)-- +(-0.2,-0.2) +(0.2,-0.2)-- +(-0.2,0.2);
\end{tikzpicture}
\label{fig:HamA}
}
\hfill
\subfloat[]{
\begin{tikzpicture}[scale=0.5]
\draw[thick,->] (-1,-1) -- (8,-1) node[anchor=west]{$x$};
\draw[thick,->] (-1,-1) -- (-1,8) node[anchor=south]{$t$};

\draw[blue!50,dashed] (2,0) -- (2,4) -- (1,4) -- (1,8);
\draw[blue!50,dashed] (3,0) -- (3,8);
\draw[blue!50,dashed] (7,0) -- (7,2) -- (5,2) -- (5,6) -- (4,6) -- (4,8);
\draw[blue!50,dashed] (6,5) -- (6,8);

\draw[red,very thick] (7,2) -- (7,6) -- (5,6) -- (5,8);
\draw[red,very thick] (4.5,0) -- (4.5,4) -- (2,4) -- (2,8);

\fill[blue!50] (1,4) circle(5pt);
\fill[blue!50] (5,2) circle(5pt);
\fill[blue!50] (4,6) circle(5pt);
\fill[blue!50] (6,5) circle(5pt);

\draw[red!50!black,ultra thick] (2,4) +(0.2,0.2)-- +(-0.2,-0.2) +(0.2,-0.2)-- +(-0.2,0.2);
\draw[red!50!black,ultra thick] (5,6) +(0.2,0.2)-- +(-0.2,-0.2) +(0.2,-0.2)-- +(-0.2,0.2);
\draw[red!50!black,ultra thick] (7,2) +(0.2,0.2)-- +(-0.2,-0.2) +(0.2,-0.2)-- +(-0.2,0.2);
\end{tikzpicture}
\label{fig:HamB}
}
\caption{\protect\subref{fig:HamA}~The dynamics of the particles in the Hammersley process with some initial configuration of the particles. The time flows from bottom to top.
     \protect\subref{fig:HamB}~The second line of the multi-line Hammersley process.}
\end{figure}

The information about the sequence $w=(w_1,\dots,w_\ell)$ can be encoded by a
collection of points $(w_1,1), \dots, (w_\ell,\ell)$ on the plane (marked as
small disks on \cref{fig:HamA}). The time evolution of the bottom row of the
insertion tableau in the process of insertions \eqref{eq:insertion} can be
encoded by the time evolution of a collection of particles on the real line
(their trajectories are marked on \cref{fig:HamA} as blue zig-zag lines) which is subject to the
following dynamics. When we have reached one of the disks~$(x,t)$ (translation:
\emph{at time $t$, when a number $x$ is inserted into the bottom row of the insertion
tableau\dots}) one of the following happens: (i)~a~particle, which is first to the
right of $x$, jumps left to $x$ (translation: \emph{\dots the~newly inserted number $x$ bumps from the bottom row the
smallest number which is bigger than $x$}), or (ii)~a~new particle is created in
$x$ (translation: \emph{the~number $x$ is appended at the end of the bottom row}), see
\cref{fig:HamA} for an illustration.

If the locations of the disks on the upper halfplane are random, sampled
according to the Poisson point process on $I\times \R_+$ (for some specified set
$I\subseteq \R$) we obtain in this way the celebrated \emph{Hammersley process}
on $I$ \cite{Hammersley1972,Aldous1995}.

\medskip

The information about all bumpings from the bottom row of the insertion tableau
can be encoded by the dual corners \cite{Ferrari2009} (marked on
\cref{fig:HamA} by red X crosses). These crosses are used as an input for the
dynamics of the second row of the insertion tableau in an analogous way as the
disks were used for the dynamics of the bottom row, see \cref{fig:HamB}. In
other words, the output of the Hammersley process (which will be the first line
of the~\emph{multi-line Hammersley process} which we will construct) is used as
the input for the second line of the \emph{multi-line Hammersley process}. 

This procedure can be iterated; in this way the dynamics of all rows of the insertion
tableau is fully encoded by \emph{the multi-line Hammersley process}
\cite{Ferrari2009}.
The
name is motivated by the analogy with the tandem queues where the happy
customers who exit one waiting line are the input for the~second line.

\subsection{Limit distribution of the multi-line Hammersley process}
\label{sec:multi-line-hammersley-process}

As we already mentioned, the entries of the bottom row of the insertion tableau
can be interpreted as positions of the particles in (the de-Poissonized version~of) the Hammersley interacting particle process
on the unit interval $[0,1]$. Therefore the
following result (\cref{coro:localP} below) is a generalization of the result of Aldous and
Diaconis \AD which concerned only the special case
$k=0$ of the single-line Hammersley process (in the Poissonized setup). For a more detailed
discussion of the link between these results see \cref{sec:link-aldous-diaconis}.

The general case~$k\geq 0$ can be interpreted as a statement about the
convergence of the multi-line version of the Hammersley process
\emph{on the the unit interval $[0,1]$} to its stationary
distribution \emph{on the whole real line $\R$} which was calculated by Fan and
Seppäläinen \cite[Theorem 5.1]{Fan2020}.

Note that in his original paper \cite[Section 9]{Hammersley1972} Hammersley
considered the particle process with a discrete time parameter indexed by
non-negative integers. Slightly confusingly, this process with the modern
terminology would be referred to as \emph{the de-Poissonized version of the Hammersley
    process} (as opposed to \emph{the Hammersley process} in which the time is
continuous and the input is given by the Poisson point process on the quarterplane). It~follows that the
setup which we consider in \cref{coro:localP} coincides with the one from the
original paper of Hammersley. The special case $k=0$ of \cref{coro:localP}
was conjectured already by Hammersley \cite[\emph{``assumption~$\alpha$''} on
page~371]{Hammersley1972} who did not predict the exact value of the intensity of the Poisson process.

\begin{corollary}
    \label{coro:localP}

Let $\xi=(\xi_1,\dots\xi_n)$ be a sequence of independent,
identically distributed random variables with the uniform distribution $U(0,1)$
on the unit interval $[0,1]$ and let
\[ \left[ P^{(n)}_{x,y} \right]_{y\in\N_0,\; 0\leq x< \lambda^{(n)}_y} = P(\xi_1,\dots,\xi_n)\]
be the corresponding insertion tableau; we denote by $\lambda^{(n)}$ its shape.

    For any integer $k\in\N_0$ and any real number $0<w<1$ 
    the~collection of $k+1$ random sets $\myPset_0,\dots,\myPset_k$ with
\begin{equation}
     \label{eq:myPset}
       \myPset_y:= \left\{ \sqrt{n} \left( P^{(n)}_{x,y}-w\right) \ : \ 0\leq x < \lambda^{(n)}_y \right\}   
\end{equation}    
converges in distribution, as $n\to\infty$, to a family of $k+1$ independent Poisson point
    processes on $\R$ with the intensity $\frac{1}{\sqrt{w}}$.
    
    The above statement remains true for $w=1$ but the limit in this case is 
		a~family of $k+1$ independent Poisson point processes on the negative
halfline $\R_-$ with the unit intensity.
\end{corollary}

The key ingredient of
the proof is to use some symmetries of the RSK algorithm which allow to
interchange the roles of the insertion tableau and the recording tableau, 
see \cref{sec:proof-localP} for the details.

\medskip

We were inspired to state \cref{coro:localP,coro:localQ} by the work of
Azangulov \cite{AzangulovThesis2020} who studied fluctuations of the last entry
in the bottom row of $P^{(n)}$ around $w=1$; more specifically he proved that
the (shifted and rescaled) last entry in the bottom row
\[ \sqrt{n} \left( 1 - P^{(n)}_{0,\lambda_0^{(n)}-1} \right)\] converges in
law to the exponential distribution $\operatorname{Exp}(1)$.

\subsection{The idea behind the proof: the link between the Ulam's problem and the Hammersley's assumption $\alpha$}
\label{sec:link-ulam-hammersley}

The key idea behind the proof of \cref{thm:poisson-growth-Plancherel} lies in
the intimate interplay between \emph{the Ulam's problem} and \emph{the Hammersley's
assumption $\alpha$} which was already subject to investigation by several
researchers in this field.

\subsubsection{Ulam's problem}
\label{sec:ulamsproblem}

Recall that Ulam \cite{Ulam1961} asked about the value of the limit
\begin{equation}
    \label{eq:ulam}
     c = \lim_{n\to\infty} \frac{\E \lambda_0^{(n)}}{\sqrt{n}}. 
\end{equation}
The first solution to this problem consisted of two components: proving the lower bound
$c\geq 2$ and the upper bound $c\leq 2$; interestingly these two components have quite
different proofs. 

\smallskip

The lower bound $c\geq 2$ was proved independently by Logan and Shepp \cite{Logan1977} as well
as by Vershik and Kerov \cite{Versik1977} by finding explicitly the limit shape
of typical random Young diagrams distributed according to the Plancherel
measure. Both proofs were based on the hook-length formula for the number of
standard Young tableaux of prescribed shape and finding the minimizer of the
corresponding functional. An alternative approach which avoids the variational
calculus is to use the results of Biane \cite{Biane:2001aa} in order to show
that the (scaled down) \emph{transition measure} of a Plancherel-distributed random diagram
$\lambda^{(n)}$ converges in probability to \emph{the semicircle distribution}
and to deduce that the probability of the event $\lambda^{(n)}_0 < (2-\epsilon)
\sqrt{n}$ converges to zero for each $\epsilon>0$. 

This lower bound $c\geq 2$ will play an
important role in our paper and we will use it in order to show
\cref{prop:asymptotic-probability} (more specifically, we use it in
\cref{lem:liminf}).

\smallskip

The upper bound $c\leq 2$ is due to Vershik and Kerov \cite{Vershik1985a}.
This upper bound plays an even more important role in our paper.
We will come back to this topic in \cref{sec:idea} below.

\subsubsection{Hammersley's assumption $\alpha$}
\label{subsec:assumption-alpha}

Hammersley formulated his \emph{assumption~$\alpha$}
\cite[page~371]{Hammersley1972} as a rather vague statement 
(\emph{``It is reasonable to assume that the distribution of the discontinuities $y_i$
    is locally homogenous and random''}) which we interpret as
a conjecture that the local
behavior of the numbers in the bottom row of the insertion tableau $P^{(n)}$
after appropriate rescaling converges to some Poisson point process
with unspecified intensity, cf.~\cref{coro:localP}. 
Hammersley also gave an informal argument which explained how
\emph{the assumption $\alpha$} would give solution to the Ulam's problem and he
correctly predicted the value of the constant $c=2$.

The first proof of a result of the flavor of \emph{the assumption $\alpha$} is
the aforementioned work of Aldous and Diaconis \AD.
Interestingly, the proof starts with two separate parts: one which happens to
give an alternative proof for the upper bound $c\leq 2$ for the constant in
the Ulam's problem, and one which happens to give an alternative proof for the lower
bound $c\geq 2 $. Finally, the combination of these two results gives 
the desired proof of \emph{the assumption~$\alpha$}.

The arguments in the aforementioned papers \cite{Hammersley1972,Aldous1995} were
based on a  probabilistic analysis of the Hammersley process viewed as an
interacting particle system (see \cref{sec:hammersley-process}) and thus were
quite different from those mentioned in \cref{sec:ulamsproblem}.

\subsubsection{The idea of the proof}
\label{sec:idea}

As we can see from the aforementioned papers \cite{Hammersley1972,Aldous1995},
the Ulam's problem and the assumption $\alpha$ are intimately related one with
another and a solution to one of them gives (at least heuristically) the
solution to the other one. From this perspective it is somewhat surprising that
the original solution to the Ulam's problem contained in the papers
\cite{Logan1977,Versik1977,Vershik1985a} did not result with a corresponding
proof of the assumption~$\alpha$ in the language of the Plancherel growth process
and random Young diagrams. The current paper fills this gap.

\medskip

Our strategy is to revisit the proof of the upper bound $c\leq 2$ which is due to Vershik and Kerov
{\begin{otherlanguage*}{russian}\cite[Раздел~3, Лемма~6]{Vershik1985a}\end{otherlanguage*}}
(see also \cite[Section 3, Lemma 6]{Vershik1985} for the English translation;
be~advised that there are \emph{two} lemmas having number 6 in this paper). With
the notations used in our paper (see the proof of
\cref{lem:single-step-plancherel-conditional}), this proof can be rephrased as
an application of the Cauchy--Schwarz inequality for a clever choice of a pair
of vectors $X$ and~$Y$ of (approximately) unit length. It is somewhat surprising that such a coarse bound as
Cauchy--Schwarz inequality gives the optimal upper bound $c\leq 2$ for the Ulam's
constant. This phenomenon is an indication that the Cauchy--Schwarz inequality is
applied here in a setting in which it becomes (asymptotically) saturated, which
implies that the vectors $X$ and~$Y$ are (approximately) multiples of one
another and therefore $X\approx Y$. 
It follows in particular that
\begin{equation} 
    \label{eq:scalar} 
    \langle X-Y,\;\; Y \rangle \approx 0. 
\end{equation}

Let $\lambda^{(0)} \nearrow \lambda^{(1)} \nearrow \cdots $ be the Plancherel
growth process. It turns out that a slight modification of the  left-hand side of \eqref{eq:scalar} has a
natural probabilistic interpretation as the total variation distance between:
\begin{itemize}
    \item the probability distribution of the Young diagram $\lambda^{(n)}$, 
    and
\item the \emph{conditional} probability distribution of the Young diagram
$\lambda^{(n)}$ \emph{under the condition} that the growth between the
diagrams $\lambda^{(n-1)}$ and $\lambda^{(n)}$ occurred in a specified row.
\end{itemize}
In particular, \eqref{eq:scalar} implies that this total variation distance converges to zero as $n\to\infty$,
see \cref{lem:single-step-plancherel-conditional} for a precise statement.

Heuristically, this means that the information about 
the number of the row~$r^{(n)}$ in which the growth
occurred between $\lambda^{(n-1)}$ and $\lambda^{(n)}$  
does not influence too
much the distribution of the resulting random Young diagram~$\lambda^{(n)}$. 
Since the Plancherel growth process is a Markov process,
this argument can be iterated to show that 
the numbers of the rows 
\begin{equation}\label{eq:rows}
    r^{(n+1)},\dots,r^{(n+\ell)} 
\end{equation}
in which the growths occur 
in the part of
the Plancherel growth process $\lambda^{(n)}\nearrow \cdots
\nearrow\lambda^{(n+\ell)}$
are approximately independent random variables, see
\cref{thm:independent-explicit}  for a precise statement. Various variants of the  assumption $\alpha$ are now simple corollaries. 
 
This approach gives some additional information --- which does not seem to
be available by the hydrodynamic approach \cite{Aldous1995,Cator2005} --- about the
asymptotic independence of the rows \eqref{eq:rows} \emph{and} the shape of the final Young diagram
$\lambda^{(n+\ell)}$, see \cref{thm:independent-explicit} for more detail. This
additional information will be essential in our forthcoming paper devoted to the
refined asymptotics of the bumping routes \cite{MMS-Bumping2020}.

\medskip

The aforementioned \cref{lem:single-step-plancherel-conditional} can be seen as
an additional step in the reasoning which was overlooked by the authors of
\cite{Vershik1985a}. The monograph of Romik \cite[Section 1.19]{Romik2015a}
contains a~more pedagogical presentation of these ideas of Vershik and Kerov; in
the following we will use Romik's notations with some minor adjustments.

\medskip

Note that, analogously as the proof of Aldous and Diaconis \AD, our proof of the
assumption $\alpha$ is based on combining two components: the one which is
related to the lower bound $c\geq 2$ in the Ulam's problem (see \cref{lem:liminf})
and the one related to the upper bound $c\leq 2$ (see the proof of
\cref{lem:single-step-plancherel-conditional}) and only combination of these two
components completes the proof.

\subsection{The link with the work of Aldous and Diaconis}
\label{sec:link-aldous-diaconis}

\newcommand{\ttt}{s}
\newcommand{\NNN}{\mathbf{N}^+}
\newcommand{\aaa}{w}

Following the notations from \cite{Aldous1995} we consider the
Hammersley process on $\R_+$ (starting from the empty configuration of the particles) 
and denote by $\NNN(x,t)$ the number of particles at
time $t$ which have their spacial coordinate $\leq x$.
Aldous and Diaconis \AD proved that for each fixed $\aaa>0$ the counting process 
\begin{equation}
    \label{eq:AD1}
     \left( \NNN(\aaa\ttt +y, \ttt) -
\NNN(\aaa\ttt,\ttt), \; y\in\R \right)
\end{equation}
converges in distribution, as $\ttt \to\infty$, to the
Poisson counting process with intensity~$\aaa^{-\nicefrac{1}{2}}$.

\subsubsection{The link of \cref{coro:localP}}

In the following we sketch very briefly the proof that the special case of \cref{coro:localP} which
corresponds to $k=0$ is equivalent to the aforementioned result of Aldous and Diaconis.

Due to the \emph{space-time scale invariance} \cite[Lemma 4]{Aldous1995},
the stochastic process \eqref{eq:AD1} has the same distribution as
\begin{equation}
    \label{eq:AD2}
    \left( \NNN \! \left(\aaa+\frac{y}{\ttt},\ttt^2\right) -
    \NNN \! \left(\aaa,\ttt^2\right), \; y\in\R \right).
\end{equation}
Let $\eta_1(\ttt^2)<\eta_2(\ttt^2)<\cdots$ denote the positions of the particles at
time~$\ttt^2$. The result of Aldous and Diaconis can be therefore rephrased as
convergence in distribution of the set of jumps of the function \eqref{eq:AD2}
which is equal to
\begin{equation}
    \label{eq:myset}
     \left\{  \ttt \cdot \left[ \eta_i \! \left(\ttt^2\right) - \aaa \right] : i\in\N \right\} 
\end{equation}
towards the Poisson point process with the intensity $\aaa
^{-\nicefrac{1}{2}}$.

For simplicity we restrict our attention to $0<\aaa<1$ 
(the general case \mbox{$\aaa>0$} can be obtained by an application of 
a slightly more involved space-time scale invariance). 
The above result does not change if we modify our setup and consider the Hammersley process on the unit
interval $[0,1]$. The number of disks (which are the input for the Hammersley
process) in the rectangle $[0,1] \times [0,t]$ is 
equal to the value $N(t)$ of the Poisson process at time~$t$. 
With the notations of \cref{coro:localP}, 
the entries $\left( P_{x,0}^{(n)} : 0\leq x < \lambda_0^{(n)} \right)$ of the bottom row 
in the insertion tableau $P^{(n)}$ can be interpreted as
the coordinates of the particles in the Hammersley process at the time when $n$
disks appeared; it follows that the set \eqref{eq:myset} is equal to
\begin{equation}
    \label{eq:myset2}
    \left\{ s \cdot \left[
P_{x,0}^{(N(s^2))} -\aaa \right] \ : \ 0\leq x < \lambda^{(n)}_0 \right\}.
\end{equation}

In this way we proved that the distribution of the random set \eqref{eq:myset2} (which
appears in a reformulation of the result of Aldous and Diaconis) is the mixture
of the probability distribution of the random set $\myPset_0$ given by \eqref{eq:myPset}
which corresponds to the bottom row, rescaled by the factor~$\frac{s}{\sqrt{n}}$. The mixture is taken over
$n:=N(s^2)$ which has the Poisson distribution $\operatorname{Pois}(s^2)$. Since the
scaling factor $\frac{s}{\sqrt{n}}$ converges in probability to~$1$ as $s\to\infty$, the scaling
is asymptotically irrelevant. 

We proved in this way that the result of Aldous and Diaconis \AD is a consequence of the special case of \cref{coro:localP} for $k=0$, obtained
by a straightforward Poissonization procedure.

\medskip

The implication in
the opposite direction is more challenging, but general de-Poissonization
techniques~\cite{Jacquet1998}
can be used to show that the result of Aldous and Diaconis implies the special
case $k=0$ of our \cref{coro:localP}.

\subsubsection{The link of \cref{thm:poisson-growth-Plancherel}}

In the following we sketch the proof that the special case of
\cref{thm:poisson-growth-Plancherel} which corresponds to $k=0$ is equivalent
to the result of Aldous and Diaconis \eqref{eq:AD1}. For simplicity we will consider only
the special case $\aaa=1$; the general case $\aaa>0$ would follow from a slightly
more complex scaling of the space-time.

\medskip

Due to the \emph{space-time scale invariance} \cite[Lemma~4]{Aldous1995},
the stochastic process \eqref{eq:AD1} in the special case $\aaa=1$ has the same distribution as
\[     \left( \NNN \! \left(  \ttt^2  +y \ttt, \, 1\right) -
    \NNN \! \left( \ttt^2, \, 1\right), \; y\in\R \right)
    \]
and, due to the \emph{space-time interchange property} \cite[Lemma~3]{Aldous1995}, the same distribution as 
\begin{equation}
    \label{eq:process}
         \left( \NNN \! \left(  1,\, \ttt^2  +t \ttt  \right) -
\NNN \! \left( 1,\, \ttt^2\right), \; t \in\R \right).
\end{equation}
By \AD the process \eqref{eq:AD1} for $\aaa=1$ or, equivalently, the process
\eqref{eq:process} converges for $\ttt\to\infty$ to the Poisson process with the unit
intensity. This result does not change if we modify our setup and consider the
Hammersley process on the unit interval $[0,1]$.

For each $t\geq 0$ the number of disks (which are the input for the Hammersley
process) in the rectangle $[0,1] \times [0,t]$ is equal to the value $N(t)$ 
of the Poisson process at time $t$. 
Thus the number of
all particles at time $t$, 
given by  $\NNN(1,t)=\lambda_0^{(N(t))}$,
 is equal to the length of the bottom row of the
insertion tableau after $N(t)$ disks appeared. In this way we proved that the
probability distribution of the process \eqref{eq:process} coincides with the
probability distribution of the process
\begin{equation}
    \label{eq:gekon}
     \left( \lambda_0^{\left( N(s^2+t s) \right) } - \lambda_0^{\left( N( s^2 )\right) } ,
\;   t\in \R 
\right).
\end{equation}

We set $n:=N(s^2)$ and 
\begin{equation}
    \label{eq:panTau}
     \tau(t)  = \frac{ N(s^2+t s) - N(s^2) }{\sqrt{N(s^2)}} 
\end{equation}
so that $ N(s^2+t s) = n_{\tau(t)}$ with the notations from \eqref{eq:nt}. In the
case when $N(s^2)=0$ and \eqref{eq:panTau} is not well-defined, we declare that
$n_{\tau(t)}= N(s^2+t s)$ by definition. Since the probability of the event that
$N(s^2)=0$ converges to zero as $s\to\infty$, this will not create problems in the following.

The random function $\tau$
converges in probability, as $s\to\infty$, to the identity map 
$t\mapsto t$ uniformly over compact subsets.
With these notations 
the
probability distribution of the process \eqref{eq:process} coincides with the
probability distribution of the process
\begin{equation}
    \label{eq:gekon2}
    \left( \lambda_0^{\left(n_{\tau(t)} \right) } - \lambda_0^{\left( n \right) } ,
    \;   t\in \R 
    \right).
\end{equation}
By taking the average over the random values of $n$ and $\tau$
it follows that the special case of  \cref{thm:poisson-growth-Plancherel} for
$k=0$ implies that \eqref{eq:gekon2} indeed converges to the Poisson process.
We proved in this way that the result of Aldous and Diaconis \AD (at least in the special case $w=1$) is
a consequence of the special case of \cref{thm:poisson-growth-Plancherel} for $k=0$, obtained
by a rather straightforward Poissonization procedure.

\medskip

The implication in
the opposite direction is more challenging, but again general de-Poissonization
techniques~\cite{Jacquet1998}
can be applied.

\subsection{Asymptotics of the bottom rows}
\label{sec:other-links}

The research related to the Ulam's problem culminated in the works of Baik, Deift
and Johansson \cite{Baik1999, Baik2000} as well as its extensions
\cite{Borodin2000,Okounkov2000,Johansson2001}. Roughly speaking, these results
say that the suitably normalized lengths of the bottom rows of a Plancherel
distributed random Young diagram converge to an explicit non-Gaussian limit which is
related to the eigenvalues of a large GUE random matrix, see the monograph
\cite{Romik2015a} for a more pedagogical introduction.

Such a non-Gaussian limit behavior for the lengths of the bottom rows is at a
sharp contrast with our \cref{thm:poisson-growth-Plancherel} which states that
the growths of the bottom rows $\Delta_n(t)$ are given by the Poisson process
which with the right scaling converges to the Brownian motion. This discrepancy
is an indication that the process $\Delta_n(t)$ considered in that theorem
\emph{cannot} be approximated by the Poisson process in the scaling as $t=t(n)\gg
1$. It would be interesting to find the scaling in which this passage from the
regime of independent growths to the non-Gaussian regime related to the random
matrices occurs. See also \cref{problem:32}.

\subsection{Possible generalization of \cref{coro:localP}: the bottom rows in the Schur--Weyl insertion tableau}
\label{sec:generalization}

\newcommand{\myletter}{w}

We suspect that \cref{coro:localP} is a special case of a conjectural result
which would hold in a much wider generality. In the current section we will present
the details.

\subsubsection{Schur--Weyl insertion tableau}
For given positive integers $d$ and $n$ let $\myletter=(\myletter_1,\dots,\myletter_n)$ be a
sequence of independent, identically distributed random variables with the
uniform distribution on the discrete set $\{1,\dots,d\}$ with $d$ elements; we
denote by $P=P(\myletter_1,\dots,\myletter_n)$ the corresponding insertion
tableau. This random tableau appears naturally in the context of \emph{the
    Schur--Weyl duality} which we review in the following. The tensor product $\big(
\mathbb{C}^d \big)^{\otimes n}$ has a natural structure of a $\Sym{n} \times
\operatorname{GL}_d(\mathbb{C})$-module; the irreducible components are naturally indexed by Young diagrams.
\emph{The Schur--Weyl measure} is defined as the probability distribution on
Young diagrams which corresponds to sampling a random irreducible component of
$\big( \mathbb{C}^d \big)^{\otimes n}$ with the probability proportional to the
dimension of the component, see \cite{Meliot:2011vr} for the details. The
probability distribution of the shape of $P$ coincides with the Schur--Weyl
measure; for this reason we will call the random tableau $P$ itself \emph{the Schur--Weyl
    insertion tableau}.
 
In the scaling when $d,n\to\infty$ tend to infinity in such a way that $d\simeq c
\sqrt{n}$ for some $c> 0$, there is a law of large numbers for \emph{the
    global form} of the scaled down insertion tableau $P$, see \cite[Remark 1.6 and
Section 1.7.3]{Matsumoto:2020tb}.

\medskip

In the following we will concentrate on another aspect of the Schur--Weyl
insertion tableau~$P$, namely in the entries which are located in the bottom row,
near its end. 
For an integer $i\in\{1,\dots,d\}$ we denote by $M_i=M_i(P)$ 
the number of occurrences of $i$
in the bottom row of $P$; 
our problem is to understand the joint
distribution of the family of random variables $M_i$ indexed by $i\in I$ in some
interval of the form 
\begin{equation}
    \label{eq:my-interval}
     I=\{d+1-\ell,d+2-\ell,\dots,d\}
\end{equation}
for some choice of an integer $\ell\geq 1$.

 There are two interesting limits which one can consider in this setup,
namely the one when the size of the alphabet $d\to\infty$ tends to infinity, and
the one when the length of the sequence $n\to\infty$ tends to infinity.  In the
following we will review these two limits and discuss what happens when these two
limits are iterated.

\subsubsection{The limit $d\to\infty$}
\label{sec:limit-d}

When $d\to\infty$ and the length $n$ of the sequence is fixed, 
the probability that $\myletter_1,\dots,\myletter_n$ are all distinct converges
to $1$. In the following we consider the conditional probability space for which
$\myletter_1,\dots,\myletter_n$ are all distinct; it is easy to check that such a
conditional joint distribution of the ratios
$\frac{\myletter_1}{d},\dots,\frac{\myletter_n}{d}$ converges to that of a tuple
$\xi_1,\dots,\xi_n$ of independent random variables with the uniform distribution
$U(0,1)$. It follows that the (conditional, and hence unconditional as well)
probability distribution of the tableau $\frac{1}{d} P$ (which is obtained from
$P$ by dividing each entry by $d$) converges to the probability distribution of the insertion tableau
$P(\xi_1,\dots,\xi_n)$.

We see that the iterated limit in which we \emph{first} take the limit as the
size of the alphabet $d\to\infty$ tends to infinity and \emph{then} the limit
$n\to\infty$ as the length of the sequence tends to infinity, is the one in which we
recover \cref{coro:localP}.  Heuristically, we can expect that if $n\to\infty$
and $d=d(n)\gg n^2$ tends to infinity fast enough, 
each integer from the interval $I$ given by \eqref{eq:my-interval}
of moderate length $\ell=O\left( \frac{d}{\sqrt{n}} \right)$
will not appear more than once
(except for asymptotically negligible probability); and that
the probability that a given integer $i\in I $ will appear in the bottom row is
of order 
\[
q =  \frac{n}{d} \cdot \frac{1}{\sqrt{n}}  = \frac{\sqrt{n}}{d}\approx \frac{\sqrt{n}}{d+\sqrt{n}}\ll 1
\] 
which is approximately the product of the probability that $i$ belongs to the sequence $\myletter$ (which is roughly $\frac{n}{d}$) and the probability that 
a given large entry of the sequence $\myletter$ will be in the bottom row (which is roughly $\frac{1}{\sqrt{n}}$ by \cref{coro:localP}).

Such a Bernoulli probability distribution on the set $\{0,1\}$ with the success
probability $q$ can be approximated by the geometric distribution on the set of
non-negative integers with the parameter $p=1-q$. We formalize this informal discussion as the
following conjecture.

\begin{conjecture}
    \label{conjecture:geometric}
    Let $d=d(n)$ and $\ell=\ell(n)$ be as above. 
    
    Then
    the total variation distance between:
    \begin{itemize}
        \item the random vector of the multiplicities
        \[ \left(  M_i  : i\in I\right), \]  
        and
        \item 
        the collection of $\ell$ independent random variables, each with the geometric distribution $\operatorname{Geo}(p)$ with the parameter
        \[p=1-\frac{\sqrt{n}}{d+\sqrt{n}}=\frac{d}{d+\sqrt{n}}\]
    \end{itemize}
    converges to zero, as $n\to\infty$.
\end{conjecture}

In the following we will discuss whether this conjecture is plausible also in
some other choices of the scaling for the parameters $d$ and $n$.

\subsubsection{The limit $n\to\infty$} 
\label{sec:limit-n}

For the purposes of the current section by \emph{the $d\times d$ GUE random matrix}
we understand the random matrix with the Gaussian distribution supported on
Hermitian $d\times d$ matrices, which is invariant under conjugation by unitary
matrices, and normalized in such a way that the variance of each diagonal entry
is equal to $d$. %
The \emph{traceless GUE random matrix} is obtained from the above random matrix by
subtracting a multiple of the identity matrix in such a way that the trace of the
outcome is equal to zero.

\newcommand{\restab}[1]{\lambda^{(\downarrow #1)}}
For
$m\in\{0,1,\dots,d\}$ we denote by 
\[ \restab{m}= \big(\restab{m}_0,\dots, \restab{m}_{m-1}  \big) = \operatorname{sh} P^{\leq m}  \]
the shape of the tableau 
$P^{\leq m}$ which is obtained from the Schur--Weyl random insertion tableau $P$ by
keeping only the boxes which are at most $m$; in particular $\restab{d}$ is
the shape of $P$. 
Recall that $P$ and, as a consequence, $\restab{m}$ as well, depend implicitly on the choice of the positive integers $d$ and $n$.
In the following we fix the value of the integer $d\geq 1$ and let $n$ vary.
 Johansson \cite[Theorem 1.6]{Johansson2001}
proved that the distribution of the random vector %
    \[  
     \frac{d \restab{d} - n}{\sqrt{n}}
     \]
converges, as $n\to\infty$, to the joint distribution of the eigenvalues 
$\operatorname{spec} X$ of the
\emph{traceless $d\times d$ GUE random matrix} $(X_{ij})_{1\leq i,j\leq d}$ . 
This result can be further
extended using the ideas of Kuperberg \cite{Kuperberg:2002tq}; 
one can show that the joint distribution of the random vectors
\[  \frac{d \restab{d} - n}{\sqrt{n}}, \quad 
\frac{d \restab{d-1} - n}{\sqrt{n}}, \quad \dots, \quad
\frac{d \restab{1} - n}{\sqrt{n}}\]
converges in distribution to the joint distribution of the eigenvalues of $X$,
together with the eigenvalues of the minors of $X$ obtained by iterative removal of the
last row and the last column:
\begin{equation}
    \label{eq:spectre}
     \operatorname{spec} (X_{ij})_{1\leq i,j\leq d}, \quad
\operatorname{spec} (X_{ij})_{1\leq i,j\leq d-1}, \quad \dots, 
\quad \operatorname{spec} (X_{ij})_{1\leq i,j\leq 1}.
\end{equation}
Since the content of the current section is mostly heuristic, we skip the details
of the proof; the key point is to use \cite[Corollary~5.2]{Collins:2009wg}.

With the above notations, for $m\in\{2,\dots,d\}$ we have that 
$M_m=\restab{m}_0-\restab{m-1}_0$
is the number of entries in the bottom row of 
the tableau $P$ which are equal to $m$. Also, for $m\in\{1,\dots,d\}$ 
let $\mu^{(m)}=\max \big( \operatorname{spec} (X_{ij})_{1\leq i,j\leq m}\big)$ be
the largest eigenvalue of the minor $(X_{ij})_{1\leq i,j\leq m}$. The
aforementioned result implies, in particular, that the probability distribution
of the random vector
\begin{equation}
    \label{eq:bottom-row}
     \frac{d}{\sqrt{n}}\ \big( M_d, \; M_{d-1}, \; \dots, \; M_2 \big) 
\end{equation}
which describes the content of the bottom row of $P$ (after disregarding all entries equal to $1$)
converges, as $n\to\infty$, to the joint distribution of the entries of the random vector
\begin{equation}
    \label{eq:largest-diff}
    \left(  \mu^{(d)}-\mu^{(d-1)}, \ \ \mu^{(d-1)}-\mu^{(d-2)},
 \ \ \dots,  \ \ \mu^{(2)}-\mu^{(1)} \right). 
\end{equation}

Gorin and Shkolnikov \cite[Corollary~1.3]{Gorin2017} studied the asymptotics
$d\to\infty$ of the probability distribution of any prefix of a fixed length $\ell$ of
the random vector \eqref{eq:largest-diff}. In this scaling the difference between
\emph{the GUE random matric} and \emph{the traceless GUE random matrix}
turns out to be irrelevant. 
More specifically, they proved that for each fixed value of an integer $\ell\geq 1$ the joint distribution of the random variables  
\[    \left(  \mu^{(d)}-\mu^{(d-1)},\ \  \mu^{(d-1)}-\mu^{(d-2)},\ \ 
    \dots, \ \ \mu^{(d+1-\ell)}-\mu^{(d-\ell)} \right) \]
converges to the distribution of $\ell$ independent random variables, each with the
exponential distribution $\operatorname{Exp}(1)$.

By combining these results it follows that for each fixed integer $\ell\geq 1$
the probability distribution
of the random vector
\[    \frac{d}{\sqrt{n}}\ \big( M_d, \; M_{d-1}, \; \dots, \; M_{d+1-\ell} \big) \]
converges to that of independent exponential random variables
in the iterated limit in which we \emph{first} take the
limit $n\to\infty$ as the length of the sequence tends to infinity, and \emph{then}
the limit as the size of the alphabet $d\to\infty$ tends to infinity.
Since the exponential distribution arises naturally as a limit of the geometric distribution $\operatorname{Geo}(p)$ in the scaling as $p\to 0$, 
we suspect that the following stronger result is true.
\begin{conjecture}
\cref{conjecture:geometric} remains true if $d=d(n)$ is a sequence of positive integers which tends to
infinity in a sufficiently slow way and $\ell\geq 1$ is a fixed integer.
\end{conjecture}

\subsubsection{The joint limit $d,n\to\infty$}

Heuristically, each of the iterated limits considered at the very end of the
\cref{sec:limit-d,sec:limit-n} can be seen as a limit in which both variables
$d,n\to\infty$ tend to infinity in such a way that one of these variables grows
much faster than the other one. With these two extreme cases covered one can wonder whether \cref{conjecture:geometric} holds true in general. 

\begin{conjecture}
    \label{conjecture:general}
    \cref{conjecture:geometric} remains true if $d=d(n)$ is an arbitrary sequence
of positive integers which tends to infinity and $\ell=\ell(n)$ is a sequence
of positive integers which is either constant or tends to infinity in such a
way that $\ell=O\left( \frac{d}{\sqrt{n}} \right)$.
\end{conjecture}

Particularly interesting is the \emph{balanced scaling} in which $n$ and $d$ tend
to infinity in such a way that $\frac{d}{\sqrt{n}}$ converges to some limit so
that the law of large numbers applies to the shape of $P$ \cite[Section
3]{Biane:2001aa} as well as to the tableau $P$ itself \cite[Remark 1.6 and
Section 1.7.3]{Matsumoto:2020tb}.  \cref{conjecture:general} concerns the entries
in the bottom row; we conjecture that an analogous result for any fixed number of
bottom rows remains true.

\section{Estimates for the total variation distance}
\label{sec:application-thm-poisson}

Our main tool for proving the main results of the paper is
\cref{thm:independent-explicit}. It
gives an insight into the way in which the first rows of a Young diagram
develop in the~Plancherel growth process  (thanks to this part we will have
\cref{thm:poisson-growth-Plancherel} as a straightforward corollary), together
with the information about the global shape of the Young diagram. This latter
additional information will be key for the developments in our forthcoming paper \cite{MMS-Bumping2020}.

\subsection{Total variation distance}

Suppose that $\mu$ and $\nu$ are probability measures on the same discrete set $S$. Such
measures can be identified with real-valued functions on $S$. 
We define the \emph{total
    variation distance} between the measures $\mu$ and $\nu$
\begin{equation} 
\label{eq:TVD1}
\delta( \mu, \nu) := \frac{1}{2} \left\| \mu-\nu \right\|_{\ell^1}  = \max_{X\subset S} \left| \mu(X) - \nu(X) \right|
\end{equation}
as half of their $\ell^1$ distance as functions. 
If $X$ and $Y$ are two random variables with values in the same discrete set
$S$, we define their total variation distance $\delta(X,Y)$ as the total
variation distance between their probability distributions.

\medskip

Several times we will use the following simple lemma. 

\newcommand{\TVDindependent}{\cref{lem:TVD-independent}\ref{item:TVD-independent}\xspace}

\begin{lemma}
    \label{lem:TVD-independent} 

\

\begin{enumerate}[label=(\alph*)]
    
    \item \label{item:TVD-factorize}
    Let $X=(X_1,X_2)$ and $Y=(Y_1,Y_2)$ be random vectors with independent
coordinates and such that their first coordinates have equal distribution: $X_1 \overset{d}{=} Y_1$. Then the total variation
distance between the vectors is equal to the total variation distance between
their second coordinates:
    \[ \delta(X,Y)= \delta(X_2,Y_2).\]
    
    \item \label{item:TVD-independent}
    Let $X=(X_1,\dots,X_\ell)$ and
$Y=(Y_1,\dots,Y_\ell)$ be random vectors with independent coordinates. Then
the total variation distance between the random vectors is bounded by the sum
of the coordinate-wise total variation distances:
    \[ \delta(X,Y) \leq \sum_{1\leq i \leq \ell} \delta(X_i,Y_i).\]
    
   \item \label{item:TVD-convolution}
    
    If $\mu_1,\dots,\mu_\ell$ and $\nu_1,\dots,\nu_l$ are discrete measures on
some vector space then the total variation distance between their convolutions is bounded by 
the sum of the summand-wise total variation distances:
\[  \delta\left( \mu_1 \ast \cdots \ast \mu_\ell, \; \nu_1 \ast \cdots \ast \nu_\ell \right)
\leq \sum_{1\leq i \leq \ell} \delta(\mu_i,\nu_i).
\]
\end{enumerate}    

\end{lemma}
This result seems to be folklore wisdom \cite{1558845}, nevertheless we failed to find a conventional reference and we provide a proof below.

\begin{proof}
For part \ref{item:TVD-factorize} we view the total variation distance $ \delta(X,Y)$ as half of
the appropriate $\ell^1$ norm. This double sum factorizes thanks to independence.

\medskip

For the part \ref{item:TVD-independent} we consider a collection of
random vectors given by 
\[Z^i= (Y_1,\dots,Y_i,X_{i+1},\dots,X_\ell) \qquad \text{for
$i\in\{0,\dots,\ell\}$.}\] 
In this way $Z^0=X$ and $Z^\ell=Y$; the neighboring
random vectors $Z^{i-1}$ and $Z^{i}$ differ only on the $i$-th coordinate. By the
triangle inequality
\[ \delta(X,Y) \leq  \sum_{1\leq i \leq \ell} \delta(Z^{i-1},Z^{i})=
\sum_{1\leq i \leq \ell } \delta(X_i,Y_i),
 \]
 where the last equality is a consequence of part \ref{item:TVD-factorize}.
 
\medskip

In order to prove part \ref{item:TVD-convolution} we shall use
\ref{item:TVD-independent} with the special choice that $X_i$ is a random
variable with the distribution $\mu_i$ and $Y_i$ is a random variable with the
distribution $\nu_i$. 
An application of the same measurable map to both arguments
$X=(X_1,\dots,X_\ell)$ and $Y=(Y_1,\dots,Y_\ell)$
cannot increase the total variation distance between them, so
\[ \delta\left( \mu_1 \ast \cdots \ast \mu_\ell, \; \nu_1 \ast \cdots \ast \nu_\ell \right)=\delta( X_1+\cdots+ X_\ell, Y_1+\cdots+Y_\ell) \leq \delta(X,Y). \]
The application of part \ref{item:TVD-independent} to the right-hand side completes the proof.
\end{proof}

\subsection{Growth of rows in Plancherel growth process}

\label{sec:plancherel-growth-process-independent}

Let us fix an integer $k\in\N_0$. We define the finite set $\mathcal{N}=\{0,1,\dots,k,\infty\}$
which can be interpreted as the set of the natural numbers from the perspective
of a person who cannot count on numbers bigger than $k$ (for example, for $k=3$
we would have \emph{``zero, one, two, three, many''}).

Let $\lambda^{(0)}\nearrow \lambda^{(1)} \nearrow \cdots$ be the Plancherel growth
process. For integers $n\geq 1$ and $r\in \N_0$ we denote by $E^{(n)}_r$ the
random event which occurs if the~unique box of the skew diagram $\lambda^{(n)} /
\lambda^{(n-1)}$ is located in the row with the index $r$.
For $n\geq 1$ we define the random variable
$R^{(n)}$ which takes values in $\mathcal{N}$ and which is given by
\begin{equation*}
R^{(n)} = 
\begin{cases}
r      & \text{if the event $E^{(n)}_r$ occurs for $0\leq r\leq k$},\\
\infty  & \text{if the event $E^{(n)}_r$ occurs for some $r>k$},
\end{cases}
\end{equation*}
and which --- from the perspective of the aforementioned person with limited
counting skills --- gives the number of the row in which the growth occurred.

\medskip

\newcommand{\nzero}{m}

Let $\ell=\ell(\nzero)$ be a sequence of non-negative integers such that 
\[ 
\ell=O\left(\sqrt{\nzero} \right).
\]
For a given integer $\nzero\geq (k+1)^2$ we focus on the
specific part of the Plancherel growth process
\begin{equation}
\label{eq:focus} 
\lambda^{(\nzero)}\nearrow \cdots \nearrow\lambda^{(\nzero+\ell)}.
\end{equation}
We will encode some partial information about the growths of the rows as well
as about the final Young diagram in \eqref{eq:focus}
by the random vector
\begin{equation}
\label{eq:vn0}
V^{(\nzero)}= 
\left( 
R^{(\nzero+1)}, \: \dots , \: R^{(\nzero+\ell)}, \:
\lambda^{\left( \nzero+\ell\right)} \right) \inrv \mathcal{N}^{\ell}
\times \diagrams.   
\end{equation}

We also consider the random vector
\begin{equation}
\label{eq:ovV-def} 
\overline{V}^{(\nzero)}= 
\left( 
\overline{R}^{(\nzero+1)}, \: \dots , \: \overline{R}^{(\nzero+\ell)}, \:
\overline{\lambda}^{\left( \nzero+\ell\right)} \right) \inrv \mathcal{N}^{\ell}
\times \diagrams   
\end{equation}
which is defined as a sequence of independent random variables; the random
variables $\overline{R}^{(\nzero+1)},\dots,\overline{R}^{(\nzero+\ell)}$ have the
same distribution given by
\begin{align}
    \label{eq:distribution-rbar} 
\PP\left\{ \overline{R}^{(\nzero+i)}= r \right\} &= \frac{1}{\sqrt{\nzero}} 
& \text{for }r\in\{0,\dots,k\},
\\
 \label{eq:distribution-rbar2} 
\PP\left\{ \overline{R}^{(\nzero+i)}= \infty \right\} &= 1- \frac{k+1}{\sqrt{\nzero}} 
\end{align}
and $\overline{\lambda}^{\left(
    \nzero+\ell\right)}$ is distributed according to Plancherel measure
$\Plancherel_{\nzero+\ell}$;
in particular the random variables ${\lambda}^{\left(
    \nzero+\ell\right)}$ and $\overline{\lambda}^{\left(
    \nzero+\ell\right)}$ have the same distribution.

\medskip

Heuristically, the following result states that when Plancherel growth process
is in an~advanced stage and we observe a relatively small number of its
additional steps, the growths of the bottom rows occur approximately like
independent random variables. Additionally, these growths do not affect too
much the final shape of the Young diagram.

\begin{theorem}
\label{thm:independent-explicit} 

With the above notations, for each fixed $k\in\N_0$ the total variation
distance between $V^{(\nzero)}$ and $\overline{V}^{(\nzero)}$ converges to
zero, as $\nzero\to\infty$; more specifically
\begin{equation}\label{eq:tvd}
     \delta\left(V^{(\nzero)}, \overline{V}^{(\nzero)}\right) =
             o\left(\frac{\ell}{\sqrt{\nzero}}\right). 
\end{equation}
\end{theorem}
    
The proof is postponed to \cref{proof-of-thm:independent-explicit}; in the
forthcoming
\crefrange{sec:proof-of-independent-explicit-begins}{sec:proof-of-independent-explicit-ends} 
we will gather the tools which are necessary for this goal. In fact, we conjecture that the following stronger result is true.
\begin{conjecture}
    \label{conj:optimal}
The distance \eqref{eq:tvd} is equal to
\[ O\left( \frac{\ell}{m^{\frac{5}{6}}} \right). \]      
\end{conjecture}
We will discuss this stronger result in \cref{sec:optimal}.

\subsection{Asymptotics of growth of a given row}
\label{sec:proof-of-independent-explicit-begins}
\label{sec:VK-inspiration-start}

Our main result in this subsection is \cref{prop:asymptotic-probability} which
gives asymptotics of the probability of a~growth of a~given row in the Plancherel
growth process. This result is not new; it was proved by Okounkov \cite[Proposition 2]{Okounkov2000}. Nevertheless we provide an alternative (hopefully simpler) proof below. 
As a preparation, we start with some auxiliary lemmas.

\medskip

\newcommand{\Row}{K}

In the following we keep the notations from the beginning of \cref{sec:plancherel-growth-process-independent}; in particular 
 $\lambda^{(0)}\nearrow \lambda^{(1)} \nearrow \cdots$ is the Plancherel growth process.
Let $\Row\in\N_0$ be fixed. For $n\geq 1$ we define
\begin{equation}
    \label{eq:defsn}
     s^{(n)}_K = \sum_{0\leq r \leq \Row} \PP\left( E_r^{(n)} \right),
\end{equation}
i.e., $s_K^{(n)}$ is defined as the probability that the unique box of
the skew diagram $\lambda^{(n)}/\lambda^{(n-1)}$ is located in one of the rows $0,1,\dots,\Row$.

\begin{lemma}
    \label{lem:decreasing}
    For each $\Row\in\N_0$ the sequence $s^{(1)}_\Row, s^{(2)}_\Row, \dots$ is weakly decreasing.
\end{lemma}
\begin{proof}
    For $n\geq 1$ let $\mathbf{d}_n$ denote the unique box of
the skew diagram $\lambda^{(n)}/\lambda^{(n-1)}$. Then $s_K^{(n)}$ is equal to the
probability that the box $\mathbf{d}_n$ is located in one of the rows
$0,1,\dots,\Row$.

Romik and \Sniady \cite[Section 3.3]{RomikSniady-AnnPro} constructed a random sequence of
boxes $\mathbf{q}_1, \mathbf{q}_2, \ldots$ (which is ``\emph{the jeu de taquin
    trajectory in the lazy parametrization}'') such that for each $n\geq 1$ we have
equality of distributions \cite[Lemma 3.4]{RomikSniady-AnnPro}
\[ \mathbf{d}_n \overset{d}{=} \mathbf{q}_n \]
and, furthermore, each box $\mathbf{q}_{n+1}$ is obtained from the previous one
$\mathbf{q}_{n}$ by moving one node to the right, or node up, or by staying
put. In this way
\[ \left( \text{the number of the row of $\mathbf{q}_{n}$}\right)_{n\geq 1} \]
is a weakly increasing sequence of random variables. It follows that 
the~corresponding cumulative distribution functions evaluated in point $\Row$
\begin{multline*} s^{(n)}_\Row = 
\PP\big\{  \text{(the number of the row of $\mathbf{d}_{n}$)} \leq \Row
\big\}
= \\
\PP\big\{  \text{(the number of the row of $\mathbf{q}_{n}$)} \leq \Row
\big\}
\end{multline*}
form a weakly decreasing sequence, which completes the proof. 
\end{proof}

\begin{lemma}
    \label{lem:s-bound}
For each $\Row\in \N_0$ and $n\geq 1$
\[ s^{(n)}_\Row \leq \frac{\Row+1}{\sqrt{n}}. \]
\end{lemma}
\begin{proof}
The monograph of Romik \cite[Section 1.19]{Romik2015a} contains the proof
(which is based on the work of Vershik and Kerov 
\cite[Section~3, Lemma~6]{Vershik1985a,Vershik1985})
of the inequality
\begin{equation}
\label{eq:upper-inequality-growth} 
\PP\left( E^{(n)}_r \right) \leq \frac{1}{\sqrt{n}} 
\end{equation}
in the special case of the bottom row $r=0$. After some minor adjustments this proof is applicable
to the general case of $r\in\N_0$
(in fact, these adjustments are explicitly explained in the proof of Eq.~\eqref{eq:VK-upperbound} later on).
The summation over
$r\in\{0,\dots,\Row\}$ concludes the proof.
\end{proof}

\begin{lemma} 
    \label{lem:liminf}
    For each $\Row\in\N_0$
\[ \liminf_{n\to\infty} \frac{s^{(1)}_\Row+\cdots+s^{(n)}_\Row}{\sqrt{n}} \geq 2(\Row+1). \]
\end{lemma}    
\begin{proof}
We write $\lambda^{(n)}=(\lambda^{(n)}_0, \lambda^{(n)}_1, \dots) $ so that $\lambda^{(n)}_r$ is the length of the appropriate row of the Young diagram $\lambda^{(n)}$.
The work of Logan and Shepp
\cite{Logan1977} as well as the work of Vershik and Kerov \cite{Versik1977}
contains the proof that for each $\epsilon>0$
\begin{equation}
    \label{eq:fixedrow}
     \lim_{n\to\infty} \PP\left\{ \frac{\lambda^{(n)}_r}{\sqrt{n}} < 2-\epsilon \right\} = 0
\end{equation}
in the special case of the bottom row $r=0$.
 We will revisit this proof and
explain how to adjust it for the general case $r\in\N_0$. 

With the notations of Romik \cite[the proof of Theorem 1.23]{Romik2015a}, if
\eqref{eq:fixedrow} were not true for some $\epsilon>0$ and $r\in\N_0$, then for infinitely many
values of $n$ the corresponding function $\psi_n$ (which encodes the Young diagram $\lambda^{(n)}$
in \emph{the Russian coordinate system} \cite[Section~1.17]{Romik2015a}) would be
bounded from above by the shifted absolute value function
$|u|+ r \sqrt{\frac{2}{n}}$ on the interval $\left[ \sqrt{2}-
\frac{\epsilon}{\sqrt{2}}, \sqrt{2} \right]$. Clearly this would prevent $\psi_n$
from converging uniformly to the limit shape in contradiction to
Logan--Shepp--Vershik--Kerov limit theorem \cite[Theorem~1.22]{Romik2015a}.

\medskip

Equation~\eqref{eq:fixedrow} implies that for each $r\in\N_0$
\begin{equation}
    \label{eq:liminf}
     \liminf_{n\to\infty} \frac{\E \lambda^{(n)}_r}{\sqrt{n}} \geq 2. 
\end{equation}

\medskip

We revisit the ideas of Vershik and Kerov 
\cite[Section~3, Lemma~6]{Vershik1985a,Vershik1985}, see also
\cite[Section 1.19]{Romik2015a}.
Since the indicator of the event $E^{(m)}$ fulfills 
\[ \ind_{E_r^{(m)}} = \lambda^{(m)}_r - \lambda^{(m-1)}_r \]
it follows that 
\[ \mathbb{P}\left( E_r^{(m)} \right)= \E \ind_{E_r^{(m)}} = \E \lambda^{(m)}_r - \E \lambda^{(m-1)}_r.\]
By summing over $1\leq m\leq n$ and over $0\leq r \leq K$ it follows that
\[ s^{(1)}_\Row+\cdots+s^{(n)}_\Row = \E \lambda^{(n)}_0 + \cdots + \E \lambda^{(n)}_\Row.\]
Application of \eqref{eq:liminf} completes the proof.
\end{proof}

\begin{lemma}
    \label{lem:skasymptotics}
    For each $K\in \N_0$
    \begin{equation}
     \label{eq:s-good} 
    \lim_{n\to\infty} \sqrt{n}\; s^{(n)}_\Row =\Row+1 .
    \end{equation}
\end{lemma}
\begin{proof}    
    We will use a simplified notation and write $s^{(i)} = s^{(i)}_\Row$.

    \medskip
    
    The upper bound for the left-hand side is a consequence of \cref{lem:s-bound}.
    
     \medskip

For the lower bound, suppose \emph{a contrario} that for some $\epsilon>0$ there
exist infinitely many values of an integer $n\geq 1$ for which
\begin{equation} 
\label{eq:acontrario}
\sqrt{n}\; s^{(n)} < (1-\epsilon) (\Row+1). 
\end{equation}
Let $C>0$ be a number which will be fixed later in the proof and set $m=n+\lfloor C n\rfloor$.
\cref{lem:decreasing,lem:s-bound} imply that
\begin{multline*} 
\frac{s^{(1)}+\cdots+s^{(m)}}{(\Row+1) \sqrt{n}} \leq 
\frac{\left(s^{(1)}+\cdots+s^{(n)}\right) +  (m-n)s^{(n)}}{(\Row+1) \sqrt{n}} < \\
\frac{1}{\sqrt{n}} \left( \frac{1}{\sqrt{1}}+\cdots+\frac{1}{\sqrt{n}}\right) + 
                                            (1-\epsilon)  \frac{m-n}{n}.
\end{multline*}
On the right-hand side we may bound 
the sum $\frac{1}{\sqrt{1}}+\cdots+\frac{1}{\sqrt{n}}$
by the corresponding integral, thus 
\begin{multline*} 
    \sqrt{\frac{m}{n}} \cdot \frac{s^{(1)}+\cdots+s^{(m)}}{(\Row+1) \sqrt{m}}  =    
    \frac{s^{(1)}+\cdots+s^{(m)}}{(\Row+1) \sqrt{n}} < \\
\frac{1}{\sqrt{n}}
\int_0^{n} \frac{1}{\sqrt{x}} \dif x  + (1-\epsilon)  \frac{m-n}{n} 
=  2 + (1-\epsilon)  \frac{m-n}{n}.
\end{multline*}
Passing to the limit $n \to \infty$ for the values of $n$ 
for which \eqref{eq:acontrario} holds true, 
\cref{lem:liminf} and the~above inequality 
imply that
\[
    2 \sqrt{1+C} \leq
2 + (1-\epsilon)  C. 
\]    
However, the above inequality is \emph{not} fulfilled for any 
\[ 0< C < \frac{4 \epsilon}{(1-\epsilon)^2} \]
which completes the proof \emph{a contrario}.
\end{proof}

The following result is due to Okounkov \cite[Proposition 2]{Okounkov2000}. 
We~provide an alternative, hopefully simpler proof below.
\begin{proposition}
    \label{prop:asymptotic-probability}
    For each $i\in \N_0$
    \[ 
    \lim_{n\to\infty} \sqrt{n}\ \PP\left(  E_i^{(n)} \right) =1 .\]    
\end{proposition}
\begin{proof}
    From \eqref{eq:defsn} it follows that
    \[  \sqrt{n}\ \PP\left( E_i^{(n)} \right) = 
       \sqrt{n} \left( s^{(n)}_{i} - s^{(n)}_{i-1} \right) = 
       \sqrt{n}\ s^{(n)}_{i} - \sqrt{n}\ s^{(n)}_{i-1}. \]
    To each of the two summands on the right-hand side 
		we apply \cref{lem:skasymptotics} which completes the proof.
\end{proof}

\subsection{What happens after just one step?}
\label{sec:VK-inspiration-end}

We will prove \cref{lem:single-step-plancherel-conditional} and
\cref{lem:single-step-plancherel-conditional-c} which show that the (rough)
information about the number of the row in which the growth of a Young diagram
occurred does not influence too much the probability distribution of the
resulting Young diagram.

\begin{lemma}
    \label{lem:single-step-plancherel-conditional}
    For each $r\in\N_0$
    the total variation distance between:
    \begin{itemize}
        
        \item the probability distribution of $\lambda^{(n)}$
        (i.e.,~the Plancherel measure on the~set~$\diagrams_n$), and        
        \item the
        \emph{conditional} probability distribution of $\lambda^{(n)}$ \emph{under the condition}
        that the event $E^{(n)}_r$ occurred,
        
    \end{itemize}
    converges to zero, as $n \to\infty$.
\end{lemma}
In \cref{sec:optimal} we will discuss the conjectural rate of convergence in this result.
\begin{proof}    
    For a Young diagram $\mu=(\mu_0,\mu_1,\dots)$ and $r\in \N_0$ we denote by
    \[ \operatorname{del}_r \mu= (\mu_0,\dots,\mu_{r-1},\mu_r-1,\mu_{r+1}, \dots) \]
    the Young diagram obtained from $\mu$ by removing a single box from the~row
    with the index $r$. The Young diagram $\operatorname{del}_r \mu$ is
    well-defined only if 
   \mbox{$\mu_r> \mu_{r+1}$}.
    
    \medskip
    
    We consider the finite-dimensional vector space of real-valued functions on the set $\diagrams_{n}$ of
    Young diagrams with $n$ boxes. For any subset $A\subseteq\diagrams_{n}$ 
		we~consider the~non-negative bilinear form on this space
    \[ \langle f, g\rangle_A = \sum_{\mu\in A} f_\mu\ g_\mu \]
    and the corresponding seminorm
    \[ \| f \|_A := \sqrt{ \langle f,f\rangle_A }.\]
    An important special case is $A=\diagrams_{n}$ with the corresponding norm
$\|\cdot\|_{\diagrams_{n}}$.

    We consider two special vectors $X,Y$ in this space:
    \begin{align*}
    X_\mu & := 
    \begin{cases} 
        \frac{1}{\sqrt{(n-1)!}}\ d_{\operatorname{del}_r \mu}, 
                              &  \text{if ${\operatorname{del}_r \mu}$ is well-defined}, \\
         0 & \text{otherwise,} 
     \end{cases} 
        \\
    Y_\mu & := \frac{d_{\mu}}{\sqrt{n!}}
    \end{align*} 
   where $d_\mu$ denotes the number of standard Young tableaux of shape $\mu$.
	Observe that the vector $X_\mu$ depends on $r$.
An~important feature of these vectors is that for any set $A\subseteq
\diagrams_{n}$
   \begin{align*}
\|Y\|_{A}^2 &=   \PP\left\{ \lambda^{(n)} \in A \right\},     \\
\left\langle X,Y\right\rangle_{A} &=  
               \sqrt{n}\ \PP\left\{ \lambda^{(n)} \in A  \text{ and } E^{(n)}_r \right\}, \\
  \|X\|_{A}^2 &=      \PP\left\{ \lambda^{(n-1)} \in \operatorname{del}_r A\ \right\},                      
   \end{align*}
see \cite[Section~1.19, Proof of Lemma~1.25]{Romik2015a}.

In particular, for the special case $A=\diagrams_n$ 
\begin{equation}
    \label{eq:VK-upperbound}
        c_n:= \sqrt{n}\; \PP\left( E^{(n)}_r\right)=
      \langle X, Y \rangle_{\diagrams_{n}} \leq
       \|X\|_{\diagrams_{n}}  \cdot \|Y\|_{\diagrams_{n}} \leq 1.
\end{equation}
    By \cref{prop:asymptotic-probability} the left-hand side converges to $1$
as $n\to\infty$. 
Since $\|X\|_{\diagrams_{n}} \leq 1$ and $\|Y\|_{\diagrams_{n}} = 1$,
it follows that
    \[ \lim_{n\to\infty} c_n = \lim_{n\to\infty} \langle X, Y \rangle_{\diagrams_{n}}=
       \lim_{n\to\infty} \|X\|_{\diagrams_{n}} =
        1.\]
As a consequence, a simple calculation using bilinearity of the scalar product shows that     
    \[ \lim_{n\to\infty} \left\| c_n^{-1}  X-Y \right\|_{\diagrams_{n}} = 0. \]
    
For any $A\subseteq \diagrams_{n}$ it follows therefore that
\begin{multline*} \left| \PPcondCurly{ \lambda^{(n)} \in A }{ E^{(n)}_r } -
\PP\left\{ \lambda^{(n)} \in A  \right\} \right| = \\ =
\left|  \frac{\left\langle X,Y \right\rangle_A}{\sqrt{n}\ \PP\left( E^{(n)}_r\right)} - \left\langle Y,Y \right\rangle_A \right|
= 
\left| \left\langle c_n^{-1}   X-Y, Y\right\rangle_A  \right| \leq \\
\left\|  c_n^{-1}   X-Y \right\|_{\diagrams_{n}} \cdot
\| Y\|_{\diagrams_{n}}.
\end{multline*}
The right-hand side does not depend on the choice of $A$ and converges to zero which concludes the proof.
\end{proof}

\renewcommand{\Row}{k}

\begin{lemma}
    \label{lem:single-step-plancherel-conditional-c}
    For each $\Row\in\N_0$
    the total variation distance between:
    \begin{itemize}
        
        \item the probability distribution of $\lambda^{(n)}$
        (i.e.,~the Plancherel measure on the~set~$\diagrams_n$), and        
        \item the
        \emph{conditional} probability distribution of $\lambda^{(n)}$ \emph{under the condition}
        that the event $\left( E^{(n)}_0 \cup \cdots \cup E^{(n)}_\Row \right)^c$ occurred,
        
    \end{itemize}
    is of order $o\left(\frac{1}{\sqrt{n}}\right)$, as $n \to\infty$.
\end{lemma}
\begin{proof}
    For real numbers $x$ and $c>0$ we will denote by $x \pm c$ some unspecified real number
in the interval $[x-c,x+c]$. In the following we will use the quantity
$s_\Row^{(n)}$ defined in \eqref{eq:defsn}.
We denote
\[ F^{(n)} = \left( E^{(n)}_0 \cup \cdots \cup E^{(n)}_\Row \right)^c.\]

    Let $C_n$ be the maximum (over $r\in\{0,\dots,\Row\}$) of the total
variation distance considered in
\cref{lem:single-step-plancherel-conditional}. 
The law of total probability implies that for any 
set~\mbox{$A\subseteq \diagrams_{n}$}
    \begin{multline*} 
    \PP\left\{ \lambda^{(n)} \in A \right\} = \\
        \sum_{0\leq r\leq \Row}  \PPcondCurly{ \lambda^{(n)} \in A }{E_r^{(n)}} \PP\left( E^{(n)}_r \right)
    + \PPcondCurly{ \lambda^{(n)} \in A }{  F^{(n)} }
    \PP\left( F^{(n)}  \right) = \\
    \sum_{0\leq r\leq \Row} \left[ \PP\left\{ \lambda^{(n)} \in A \right\} \pm C_n \right] \PP\left( E^{(n)}_r \right) +
       \PPcondCurly{ \lambda^{(n)} \in A }{  F^{(n)} }
    \PP(F^{(n)})= \\
=
\PP\left\{ \lambda^{(n)} \in A \right\} s^{(n)}_\Row \pm  C_n s^{(n)}_\Row  + \PPcondCurly{ \lambda^{(n)} \in A }{  F^{(n)} }
\left( 1-  s^{(n)}_\Row \right).
    \end{multline*}
By solving the above equation for the conditional probability we get
\[ \PPcondCurly{ \lambda^{(n)} \in A }{  F^{(n)} } =
\frac{ \PP\left\{ \lambda^{(n)} \in A \right\}  \left( 1-s^{(n)}_\Row \right)\pm C_n s_\Row^{(n)} }{1-s^{(n)}_\Row}.
\]
In this way we proved that the total variation distance considered in the statement of the lemma is bounded from above by 
\begin{equation}
    \label{eq:TVD-lolcats}
      \frac{  C_n s_\Row^{(n)} }{1-s^{(n)}_\Row}.
\end{equation}

The asymptotics of the individual factors in \eqref{eq:TVD-lolcats} is provided by
\cref{lem:single-step-plancherel-conditional} (which gives $C_n=o(1)$) and by \cref{lem:skasymptotics} or, equivalently, by Okounkov's result
\cref{prop:asymptotic-probability} \big(which gives $s^{(n)}_\Row = O\left( \frac{1}{\sqrt{n}} \right) $\big); this completes the proof.
\end{proof}

\subsection{Asymptotic independence}
\label{sec:proof-of-independent-explicit-ends}

As an intermediate step towards the proof of \cref{thm:independent-explicit} we
consider a sequence of independent random variables
\begin{equation}
\label{eq:independent-vector}
\widetilde{V}^{(\nzero)}=
\left( \widetilde{R}^{(\nzero+1)},\dots,\widetilde{R}^{(\nzero+\ell)},
\widetilde{\lambda}^{\left( \nzero+\ell\right)} \right) \inrv \mathcal{N}^{\ell}
\times \diagrams
\end{equation}
which is independent with the vectors $V^{(\nzero)}$ and
$\overline{V}^{(\nzero)}$ 
(recall the definitions in \eqref{eq:vn0} and \eqref{eq:ovV-def}), 
and such that the marginal distributions of $V^{(\nzero)}$ and
\eqref{eq:independent-vector} coincide:
\begin{align*} 
\widetilde{R}^{(\nzero+i)}  &\overset{d}{=} R^{(\nzero+i)}, 
                                \qquad \text{for all } 1\leq i \leq \ell
\\
\widetilde{\lambda}^{(\nzero+\ell)} &\overset{d}{=} 
{\lambda}^{(\nzero+\ell)}.
\end{align*}
In particular, the probability distribution of $\widetilde{V}^{(\nzero)}$
depends implicitly on $k$ which is the number of the rows of Young diagrams which we observe less~$1$.

\newcommand{\mv}{n}

\begin{lemma} 
    \label{lem:make-it-a-bit-more-independent} 
    
    For each $\Row\in\N_0$ there exists a sequence $b_\mv=o\left(
\frac{1}{\sqrt{\mv}} \right)$ with the property that for all $\nzero\geq (k+1)^2$ and $\ell\geq 1$ and
$i\in\{1,\dots,\ell\}$
    \begin{multline} 
    \label{eq:independence-step} 
    \delta\bigg( \:
    \left( \widetilde{R}^{(\nzero+1)}, \dots, \widetilde{R}^{(\nzero+i-1)},
    R^{(\nzero+i)},R^{(\nzero+i+1)},\dots,R^{(\nzero+\ell)},
    \lambda^{\left( \nzero+\ell\right)} \right)
    , \\
    \left( \widetilde{R}^{(\nzero+1)}, \dots, \widetilde{R}^{(\nzero+i-1)},
\widetilde{R}^{(\nzero+i)},R^{(\nzero+i+1)},\dots,R^{(\nzero+\ell)},
\lambda^{\left( \nzero+\ell\right)} \right)
    \bigg) \\ \leq b_{\nzero+i}.
    \end{multline}
\end{lemma}
The only difference between the two random vectors considered in
\eqref{eq:independence-step} lies in the $i$-th coordinate: in the first vector
this coordinate is equal to ${R}^{(\nzero+i)}$ while in the second to
$\widetilde{R}^{(\nzero+i)}$.
\begin{proof}
For an integer $\mv\geq 1$ we define
\begin{multline*}
    b_\mv :=   \sum_{r\in \mathcal{N}} 
    \PP\left( R^{(\mv)} = r \right)  \times\\
    \sum_{\lambda\in\diagrams_{\mv}}
    \left| 
    \PPcond{ \lambda^{(\mv)}=\lambda }{ R^{(\mv)} = r } -
    \PP\left( \lambda^{(\mv)}=\lambda \right) \right|.
\end{multline*}
For the summands corresponding to $r\in\{0,\dots,k\}$ we note that the random events $\left\{ R^{(\mv)}=r \right\}$ and $E^{(\mv)}_r$ are equal and we
apply
\cref{prop:asymptotic-probability} and
\cref{lem:single-step-plancherel-conditional}. For the summand $r=\infty$ we
apply \cref{lem:single-step-plancherel-conditional-c}. This gives the desired
asymptotics  $b_\mv=o\left(
\frac{1}{\sqrt{\mv}} \right)$. In the following we will show that this
sequence indeed fulfills \eqref{eq:independence-step}.

\medskip

An iterative application of 
 \cref{lem:TVD-independent}\ref{item:TVD-factorize}
shows that
the left-hand side of~\eqref{eq:independence-step} is equal to the total
variation distance of the suffixes
   \begin{multline*} 
\delta\bigg( \:
\left(
R^{(\nzero+i)},R^{(\nzero+i+1)},\dots,R^{(\nzero+\ell)},
\lambda^{\left( \nzero+\ell\right)} \right)
, \\
\left( 
\widetilde{R}^{(\nzero+i)},R^{(\nzero+i+1)},\dots,R^{(\nzero+\ell)},
\lambda^{\left( \nzero+\ell\right)} \right)
\bigg).
\end{multline*}    

In order to evaluate the latter we consider an arbitrary set $X\subseteq
\mathcal{N}^{\ell+1-i} \times \diagrams$. We can write
   \[ X  = \bigcup_{r\in \mathcal{N}} \{ r \} \times X_r \]
   for some family of sets $X_r\subseteq \mathcal{N}^{\ell-i} \times \diagrams
$ indexed by $r\in\mathcal{N}$. 
Since the Plancherel growth process \eqref{eq:focus} is a~Markov process \cite[Sections~2.2 and 2.4]{Kerov1999},
   \begin{multline} 
   \label{eq:trololoA}
   \PP\left(
   \left(
   R^{(\nzero+i)},R^{(\nzero+i+1)},\dots,R^{(\nzero+\ell)},
   \lambda^{\left( \nzero+\ell\right)} \right) \in X \right) = \\
\shoveleft{    \sum_{r\in \mathcal{N}} \sum_{\lambda\in\diagrams_{\nzero+i}}
        \PP\left( R^{(\nzero+i)} = r\text{ and } \lambda^{(\nzero+i)}=\lambda \right)  \times} \\ 
\shoveright{        \PPcond{  
  \left( R^{(\nzero+i+1)},\dots,R^{(\nzero+\ell)},
  \lambda^{\left( \nzero+\ell\right)} \right)  
     \in X_r }{ \lambda^{(\nzero+i)}=\lambda }   =} \\
 \shoveleft{ \sum_{r\in \mathcal{N}} \sum_{\lambda\in\diagrams_{\nzero+i}}
 \PPcond{ R^{(\nzero+i)} = r \right) \PP\left( \lambda^{(\nzero+i)}=\lambda }{ R^{(\nzero+i)} = r }  \times } \\ 
 \PPcond{
   \left( R^{(\nzero+i+1)},\dots,R^{(\nzero+\ell)},
 \lambda^{\left( \nzero+\ell\right)} \right) 
  \in X_r }{ \lambda^{(\nzero+i)}=\lambda } .        
 \end{multline} 
An analogous, but simpler calculation shows that
  \begin{multline} 
     \label{eq:trololoB}
  \PP\left(
\left(
\widetilde{R}^{(\nzero+i)},R^{(\nzero+i+1)},\dots,R^{(\nzero+\ell)},
\lambda^{\left( \nzero+\ell\right)} \right) \in X \right) = 
\\
\shoveleft{ \sum_{r\in \mathcal{N}} \sum_{\lambda\in\diagrams_{\nzero+i}}
    \PP\left( R^{(\nzero+i)} = r \right) \PP\left( \lambda^{(\nzero+i)}=\lambda  \right)  \times } \\ 
\PPcond{
\left( R^{(\nzero+i+1)},\dots,R^{(\nzero+\ell)},
\lambda^{\left( \nzero+\ell\right)} \right) 
\in X_r }{ \lambda^{(\nzero+i)}=\lambda }.        
\end{multline}

The first and the third factor on the right-hand side of \eqref{eq:trololoA}
coincide with their counterparts on the right-hand side of \eqref{eq:trololoB},
and the third factor is bounded from above by $1$. It follows that the absolute
value of the difference between \eqref{eq:trololoA} and \eqref{eq:trololoB} is
bounded from above by $b_{\nzero+i}$,
as required.
\end{proof}

\subsection{Proof of \cref{thm:independent-explicit}}
\label{proof-of-thm:independent-explicit}
\begin{proof}
An iterative application of the triangle inequality combined with
\cref{lem:make-it-a-bit-more-independent} implies that
\begin{multline}
\label{eq:distanceA}
    \delta\bigg( \:
\left( {R}^{(\nzero+1)}, \dots, R^{(\nzero+\ell)},\lambda^{\left( \nzero+\ell\right)} \right)
, \\
\left( \widetilde{R}^{(\nzero+1)}, \dots, 
      \widetilde{R}^{(\nzero+\ell)}, \widetilde{\lambda}^{\left( \nzero+\ell\right)} \right) \bigg) = 
                                                             o\left( \frac{\ell}{\sqrt{\nzero}}\right).
 \end{multline}

\medskip

On the other hand, thanks to the independence of the coordinates,
\TVDindependent gives 
\begin{multline}
\label{eq:distanceB}
  \delta\bigg( \:
\left( \widetilde{R}^{(\nzero+1)}, \dots, \widetilde{R}^{(\nzero+\ell)}, 
              \widetilde{\lambda}^{\left( \nzero+\ell\right)} \right) ,\; \\
\left( 
\overline{R}^{(\nzero+1)},\dots,\overline{R}^{(\nzero+\ell)},
\overline{\lambda}^{\left( \nzero+\ell\right)} \right)\;
\bigg) \leq 
\sum_{1\leq i \leq \ell} 
\delta\left( \widetilde{R}^{(\nzero+i)}, 
\overline{R}^{(\nzero+i)} \right).
\end{multline}

\medskip

In the remaining part of the proof we will investigate the 
individual summand which corresponds to $\mv:=\nzero+i$.
We have
 \begin{multline} 
    \label{eq:TVDu}
    \delta\left( \widetilde{R}^{(\mv)}, 
\overline{R}^{(\mv)} \right) = \frac{1}{2} 
\sum_{0\leq r \leq k} 
\left| \PP\left( \widetilde{R}^{(\mv)}= r \right) - \PP\left( \overline{R}^{(\mv)}=r \right) \right| + \\
\frac{1}{2} \left| \PP\left( \widetilde{R}^{(\mv)}= \infty \right) - \PP\left( \overline{R}^{(\mv)}=\infty \right) \right|.  
\end{multline}
The equality
\[ \PP\left( \widetilde{R}^{(\mv)}= \infty \right) = 1 - \sum_{0\leq r \leq k}  \PP\left( \widetilde{R}^{(\mv)}= r \right) \]
and the analogous equality for $\overline{R}^{(\mv)}$ imply that the right-hand side of \eqref{eq:TVDu} can be bounded as follows:
 \begin{equation}
     \label{eq:TVD-boink}
    \delta\left( \widetilde{R}^{(\mv)}, 
\overline{R}^{(\mv)} \right) \leq   
\sum_{0\leq r \leq k} 
\left| \PP\left( \widetilde{R}^{(\mv)}= r \right) - \PP\left( \overline{R}^{(\mv)}=r \right) \right|.
\end{equation}
Below we will find the asymptotics of the individual summands on the right-hand side.

For any $0 \leq r \leq k$, 
by the definitions
\[ \PP\left( \widetilde{R}^{(\mv)}= r \right) = 
\PP\left( R^{(\mv)}= r \right) =
\PP\left( E_r^{(\mv)} \right) \]
and by 
\cref{prop:asymptotic-probability} we get 
the asymptotics of the probability 
$\PP\left( \widetilde{R}^{(\mv)}= r \right)$.
The asymptotics of the probability
 $\PP\left( \overline{R}^{(\mv)}= r \right)$ is given by 
its definition~\eqref{eq:distribution-rbar}. 
It follows that the total variation distance in \eqref{eq:TVD-boink} is 
of order~$o\left( \frac{1}{\sqrt{\mv}} \right)$.

\medskip

In this way we proved that 
\begin{multline}
    \label{eq:distanceB1}
    \delta\bigg( \:
    \left( \widetilde{R}^{(\nzero+1)}, \dots, \widetilde{R}^{(\nzero+\ell)}, 
    \widetilde{\lambda}^{\left( \nzero+\ell\right)} \right) ,\; \\
    \left( 
        \overline{R}^{(\nzero+1)},\dots,\overline{R}^{(\nzero+\ell)},
        \overline{\lambda}^{\left( \nzero+\ell\right)} \right)\;
        \bigg) \leq  o\left( \frac{\ell}{\sqrt{\nzero}} \right). 
\end{multline}
The triangle inequality combined with \eqref{eq:distanceA} and \eqref{eq:distanceB1} completes the
proof.
\end{proof}

The following problem was asked by Maciej Dołęga.
\begin{question}
    Plancherel growth process may be defined in terms of Schur polynomials and the
    corresponding Pieri rule. Is it possible to apply the~ideas presented in the
    current section in the context of some other growth processes on $\diagrams$
    (such as \emph{Jack--Plancherel growth process} \cite{Kerov2000}) which are
    related to other classical families of symmetric polynomials (such as
    \emph{Jack polynomials})?
\end{question}

\subsection{What is the rate of convergence?}
\label{sec:optimal}

A more refine asymptotics for the length of the given row 
\begin{equation}\label{eq:refined}
     \E \lambda_i^{(n)} = 2 \sqrt{n} + c_i n^{\frac{1}{6}} + o\left( n^{\frac{1}{6}} \right)
\end{equation}
(see \cite[Theorem 1.2]{Baik1999} for the case of the bottom row) suggests that the
probability of the event $ E_i^{(n)}$ can be approximated by the
derivative of the right-hand side of \eqref{eq:refined}, thus
\[ 1- \sqrt{n} \ \PP\left(  E_i^{(n)} \right) = O\left( n^{- \frac{1}{3}}  \right); \]
note that the same kind of argument appeared already in \cref{rem:derivative}.
If this stronger version of \cref{prop:asymptotic-probability} indeed holds
true, then also the total variation distance in
\cref{lem:single-step-plancherel-conditional} is at most $O\left(
n^{-\frac{1}{3}} \right)$. Consequently, the total variation distance in
\cref{lem:single-step-plancherel-conditional-c} as well as the sequence $b_n$
in \cref{lem:make-it-a-bit-more-independent} can be better bounded as $O\left(
n^{-\frac{1}{2}-\frac{1}{3}} \right)=O\left( n^{-\frac{5}{6}} \right)$. This would imply that \cref{conj:optimal} indeed holds true.

\section{Proofs of the main results}

\subsection{Proof of \cref{thm:poisson-growth-Plancherel}}
\label{sec:proof-thm:poisson-growth-Plancherel}

\newcommand{\mylength}{L}

    Let $\mathbf{N}(t)=\big( N_0(t), \dots, N_k(t) \big), t\in\R$, 
		be a~collection of $k+1$ independent copies of the standard Poisson process.
Let us fix some real number $c>0$; in the following we assume that $n$ is big enough so that 
$n -c \sqrt{n}  \geq (k+1)^2$.
It follows in particular that
\[ n_t =  n + \lfloor t \sqrt{n} \rfloor \qquad \text{for }t\in [-c,c].\]
We denote 
\[ \mylength=\mylength(n)=n_c-n_{-c} 
=\lfloor c \sqrt{n} \rfloor -
\lfloor -c \sqrt{n} \rfloor.\]

\begin{lemma}
    \label{lem:TVD-TGV}
For each $c>0$ the total
	variation distance between the random vector
    \begin{equation}
    \label{eq:TVDa}
\Big( \Lambda^{(n+i)} -\Lambda^{(n)} : \lfloor -c \sqrt{n} \rfloor \leq i \leq \lfloor c\sqrt{n}\rfloor \Big) 
\inrv \left( \Z^{k+1} \right)^{ \mylength+1} 
            \end{equation}
    and the corresponding random vector
    \begin{equation}
\label{eq:TVDb}
\left( \mathbf{N}\left(  \frac{i}{\sqrt{n}} \right) : \lfloor -c \sqrt{n} \rfloor \leq i \leq \lfloor c\sqrt{n}\rfloor \right)
                     \inrv \left( \Z^{k+1} \right)^{\mylength+1}
 \end{equation}
    converges to zero, as $n\to\infty$ tends to infinity. 
\end{lemma}
\begin{proof}
We consider the bijection 
\begin{multline*}  \Z^{\mylength} \ni
\Big( a_i : \lfloor -c \sqrt{n} \rfloor \leq i \leq \lfloor c\sqrt{n}\rfloor 
\textrm{ \: with \: $a_0=0$} \Big) \mapsto \\
\Big( a_{i}-a_{i-1} : \lfloor -c \sqrt{n} \rfloor < i \leq  \lfloor c\sqrt{n}\rfloor  \Big)  \in \Z^{\mylength}.
\end{multline*}
Since an application of a bijection does not change the total variation distance,
the aforementioned total variation distance between \eqref{eq:TVDa} and
\eqref{eq:TVDb} is equal to the total variation distance $\delta(A,B)$ between the
corresponding sequences of the~increments, i.e.,
\begin{align*}
A & :=\left( \Lambda^{(n+i)} - \Lambda^{(n+i-1)} :  
\lfloor -c \sqrt{n} \rfloor < i \leq \lfloor c\sqrt{n}\rfloor \right) \\
    & \nonumber
    = \bigg( \left( \ind_{R^{(n+i)}=0},  \: \dots , \: \ind_{R^{(n+i)}=k} \right)
     :  
    \lfloor -c \sqrt{n} \rfloor < i \leq \lfloor c\sqrt{n}\rfloor
    \bigg)
\intertext{and the sequence of independent random vectors}
B &:=\left( \mathbf{N}\left(\frac{i}{\sqrt{n}}\right) - \mathbf{N}\left(\frac{i-1}{\sqrt{n}}\right) :  
\lfloor -c \sqrt{n} \rfloor < i \leq \lfloor c\sqrt{n}\rfloor \right). \\
\intertext{In the following 
we use the notations from \cref{sec:plancherel-growth-process-independent} 
with $\nzero:=\left\lfloor n - c\sqrt{n}\right\rfloor$ and $\ell:=L$. 
In particular we consider the collection of random variables~$\overline{R}^{(m)}$
with the probability distribution given by \eqref{eq:distribution-rbar} 
and \eqref{eq:distribution-rbar2} for this specific value of $\nzero$.
We define}
 \overline{A} &:= \bigg( \left( \ind_{\overline{R}^{(n+i)}=0},  \: \dots , \: \ind_{\overline{R}^{(n+i)}=k} \right)
:  
\lfloor -c \sqrt{n} \rfloor < i \leq \lfloor c\sqrt{n}\rfloor
\bigg); 
\end{align*}
our strategy will be to apply the triangle inequality
\begin{equation}\label{eq:triangle}
     \delta(A,B)\leq \delta(A,\overline{A}) + \delta(\overline{A},B). 
\end{equation}

\medskip

In order to bound the first summand on the right-hand side of \eqref{eq:triangle} we apply \cref{thm:independent-explicit} for the aforementioned values of $\nzero$ and $\ell$; it follows that
\[ \delta(A,\overline{A}) \leq 
\delta\left( V^{(\nzero)}, \overline{V}^{(\nzero)} \right)=o(1).\]

\medskip

For the second summand on the right-hand side of \eqref{eq:triangle} we apply \TVDindependent 
\begin{multline*}
\delta(\overline{A},B)=
\sum_{\lfloor -c \sqrt{n} \rfloor < i \leq \lfloor c\sqrt{n}\rfloor}
\delta \Bigg( \left( \ind_{\overline{R}^{(n+i)} = 0}, \dots, \ind_{\overline{R}^{(n+i)} = k} \right) , \\  
\underbrace{\operatorname{Pois}\left(\frac{1}{\sqrt{n}} \right)
\times \cdots \times \operatorname{Pois}\left(\frac{1}{\sqrt{n}} 
\right)}_{\text{$k+1$ factors}} \Bigg)  
= O \left( \frac{1}{\sqrt{n}} \right),
\end{multline*}
where the last bound follows from a direct calculation of the total variation distance of specific probability distributions on $\N_0^{k+1}$.
\end{proof}

\begin{proof}[Proof of \cref{thm:poisson-growth-Plancherel}]
For given $t_1,\dots,t_\ell\in \R$ we select arbitrary $c> \max( |t_1|,\dots, |t_\ell|)$.
Let $n$ be big enough so that $n-c \sqrt{n} \geq (k+1)^2$. We apply
\cref{lem:TVD-TGV}; from the two random vectors which appear in this lemma we
select the coordinates which correspond to $i\in \left\{ \lfloor t_1
\sqrt{n} \rfloor , \dots, \lfloor t_\ell \sqrt{n} \rfloor \right\}$. It follows
that the total variation distance between the law of the finite-dimensional
marginal
\[ \left( \Lambda^{(n_{t_1})}-\Lambda^{(n)},  \: \dots , \:
\Lambda^{(n_{t_\ell})}-\Lambda^{(n)} \right) \] 
and the law of the appropriate marginal of $\mathbf{N}$, that is 
\begin{equation}
    \label{eq:Poisson-tilt-marginal}
     \left( \mathbf{N}\left( \frac{\lfloor t_1 \sqrt{n} \rfloor }{\sqrt{n}} \right), \; \dots , \;\mathbf{N}\left( \frac{\lfloor t_\ell \sqrt{n} \rfloor }{\sqrt{n}} \right) \right) 
\end{equation}
converges to zero as $n\to\infty$.

On the other hand, by the maximal coupling lemma
\cite[Section 8.3, Eq.~(8.19)]{Thorisson2000} 
(the definition of the total variance distance 
therein differs from ours by the factor~$2$), 
the total variation distance between \eqref{eq:Poisson-tilt-marginal}
and
\[ \left( \mathbf{N}\left( t_1 \right), \; \dots , \;\mathbf{N}\left( t_\ell \right) \right) \]
is bounded from above by  
\[ \PP\left\{ \mathbf{N}\left( \frac{\lfloor t_i \sqrt{n} \rfloor }{\sqrt{n}} \right)   
\neq
\mathbf{N}\left( t_i \right)  \text{ for some } i\in\{1,\dots,\ell\} \right\} = O\left( \frac{1}{\sqrt{n}} \right).
\]

An application of the triangle inequality for the total variation distance completes the proof.
\end{proof}

\begin{problem}
    \label{problem:32}
  Find the~precise rate of convergence in
\cref{lem:single-step-plancherel-conditional} and
\cref{thm:independent-explicit}. This convergence probably cannot be too fast
because this would imply that an analogue of
\cref{thm:poisson-growth-Plancherel} holds true also in the scaling when in
\eqref{eq:growth-Poisson} we study $t\gg 1$, and the latter would potentially
contradict the non-Gaussianity results for the lengths of the rows of
Plancherel-distributed Young diagrams
\cite{Baik1999,Baik2000,Borodin2000,Johansson2001}.
\end{problem}

\newcommand{\SetA}{A}
\newcommand{\SetB}{\overline{A}}
\newcommand{\SetC}{B}

\subsection{Proof of \cref{coro:localQ}} 
\label{sec:proof-localQ}

\begin{proof}
Let $c>0$ be arbitrary. We consider $n \in\N_0$ which is big enough so that $n- c \sqrt{n}-1 \geq (k+1)^2$. 
We denote
\[ I_n=\left\{ \left\lceil n - c\sqrt{n}\right\rceil,\dots, \left\lfloor
n+ c\sqrt{n} \right\rfloor \right\}.\]
Our strategy is to apply \cref{thm:independent-explicit} for $\nzero:=\left\lceil n -
c\sqrt{n}\right\rceil-1$ and $\ell:=|I_n|$; in particular in the following we
will use the collection of random variables $\overline{R}^{(i)}$ over $i\in I_n$
with the probability distribution given by \eqref{eq:distribution-rbar} and 
\eqref{eq:distribution-rbar2} for this specific value of $\nzero$.

We consider the following three collections of $k+1$ random subsets of~$I_n$:
\begin{itemize}
    \item the sequence  $\SetA=\left(\SetA_0,\dots,\SetA_k\right)$ with
    \begin{align*} \SetA_y  & =  \left\{ Q_{x,y} : x\in \N_0 \right\} \cap I_n \\ & = \left\{ i\in I_n : {R}^{(i)}=y \right\}  
        \end{align*}
    obtained by selecting the entries of the bottom $k+1$ rows of the recording
tableau which belong to the specified interval, 

 \item the sequence $\SetB=\left( \SetB_0,\dots, \SetB_k\right)$ with 
\[ \overline{A}_y= \left\{ i\in I_n : \overline{R}^{(i)}=y \right\}, \]

 \item the sequence $\SetC= \left(\SetC_0,\dots,\SetC_k\right)$ 
obtained by independent sampling, i.e., such that the family of random
events $\left\{ i \in \SetC_y \right\}$  indexed by $i\in I$ and $y\in\{0,\dots,k\}$
is a family of independent events, each having equal probability $\frac{1}{\sqrt{n}}$. 
\end{itemize}

\cref{thm:independent-explicit}  implies that the total variation distance
$\delta\big( \SetA, \SetB \big)$ converges to zero, as $n\to\infty$. 

The information about the sequence $\SetB$ can be alternatively encoded by the
sequence of independent random variables $ \left( v_i : i\in I_n \right)$ given by
\begin{align*} 
v_i & = \left( \ind_{i \in \SetB_0}, \dots, \ind_{i \in \SetB_k} \right) =  
    \left( \ind_{\overline{R}^{(i)}=0}, \dots, \ind_{\overline{R}^{(i)}=k} \right)  \qquad \text{for }i \in I_n \\
\intertext{Analogously the information about $\SetC$ can be encoded by 
the sequence of independent random variables $ \left( w_i : i\in I_n \right)$,
where} 
 w_i & = \left( \ind_{i \in \SetC_0}, \dots, \ind_{i \in \SetC_k} \right) \qquad \text{for }i \in I_n . 
\end{align*}
By \TVDindependent it follows that the total variation distance for the vectors
\[ \delta(\SetB,\SetC)=\delta(v,w) \leq \sum_{i\in I_n} \delta( v_i, w_i) \]
is bounded by the sum of the coordinatewise total variation distances.
The~asymptotics of the individual summand 
\[ \delta( v_i, w_i) = o\left( \frac{1}{\sqrt{n}} \right) \]
is a consequence of a direct calculation  based on the explicit form of the two probability distributions involved here.
In this way we proved that the total variation distance $\delta(\SetB,\SetC)$
converges to $0$ as $n\to\infty$.

\medskip

We apply a shift and a scaling to the random sets which form the collection~$\SetA$ and the collection $\SetC$;
it follows that the total variation distance between 
\begin{itemize}
    \item
    the collection of random subsets of $\R$
\[
\left( \left\{ \frac{Q_{x,y}-n}{\sqrt{n}} \ : \ x\in \N_0 \right\}  \cap [-c,c] : \  \  y\in\{0,\dots,k\} \right) \]
obtained by truncating \eqref{eq:setsQ}, and 

\item 
the collection of random subsets of $\R$
\begin{equation}
    \label{eq:prepoisson}
\left( \left\{ \frac{j-n}{\sqrt{n}} \ : \  j\in \SetC_y \right\}  : \  \  y\in\{0,\dots,k\} \right) 
\end{equation}
\end{itemize}
converges to zero as $n\to\infty$.

Each random set from the collection \eqref{eq:prepoisson} converges in distribution to the Poisson point process 
on the interval $[-c,c]$, see \cite[Proposition 11.3.I]{Daley-Vere-Jones2},
which concludes the proof.
\end{proof}

\subsection{Proof of \cref{coro:localP}}
\label{sec:proof-localP}

We start with an auxiliary result.

\newcommand{\cardinalityPois}{\mathfrak{n}}
\newcommand{\cardinalityBin}{\mathfrak{m}}
\newcommand{\measure}{\mathcal{M}}

\newcommand{\Const}{C}

\begin{lemma}
    \label{lem:TVD-Poisson}
    For real numbers $p,\lambda\geq 0$ and for integers $k\geq 0$ and $n$ such that $p (k+1)\leq 1$ let
$\xi_1,\dots,\xi_n$ be independent, identically distributed random variables with
the uniform distribution $U(0,1)$ and let $\overline{R}^{(1)},\dots,\overline{R}^{(n)}$ be independent, identically distributed random variables with the distribution
\begin{equation}
    \label{eq:funny-probability-dist}
\left\{    \begin{aligned}
    \PP\left\{ \overline{R}^{(i)}= r \right\} &= p 
    & \text{for }r\in\{0,\dots,k\},
    \\
    \PP\left\{ \overline{R}^{(i)}= \infty \right\} &= 1- (k+1)p,
\end{aligned}
\right.
\end{equation}
cf.~\eqref{eq:distribution-rbar} and \eqref{eq:distribution-rbar2} for an analogous distribution in a similar context.

Then for any real numbers $a,b$ such that $0\leq a\leq b\leq 1$ 
the total variation distance between:
    \begin{enumerate}[label=(\alph*)]
        \item \label{TVD:1}
        the collection of $k+1$ random sets
\begin{equation}
    \label{eq:my-collection}
             \left( [a,b]\cap \left\{ \xi_i : \overline{R}^{(i)}=y \right\}     \; : \; y\in\{0,\dots,k\} \right), 
\end{equation}
        and
        \item \label{TVD:2}
        the collection of $k+1$ independent Poisson point processes $N_0,\dots,N_k$ on the interval $[a,b]$ with the intensity $\lambda$,
    \end{enumerate}
    is bounded from above by 
    \[  (k+1)^2 n p^2 l^2  +  (k+1)\cdot \left| \lambda l - n p l \right|,\]
    where $l=b-a$ is the length of the interval. 
\end{lemma}
\begin{proof}
For each $i\in\{0,\dots,k\}$ the corresponding Poisson point process $N_i$ can be generated by the following two-step
procedure. Firstly, we sample the~number of points $\cardinalityPois_i$; it is a random
variable with the Poisson distribution with the parameter $\lambda l$, where $l:=b-a$ is the length of the interval. Secondly,
we take $\cardinalityPois_i$ independent random elements of the unit interval $[a,b]$ with
the uniform distribution. The random variables $\cardinalityPois_0,\dots,\cardinalityPois_k$ 
which
correspond to independent Poisson processes are independent.

A similar construction can be performed for the collection
\eqref{eq:my-collection} of random sets: we first sample the vector
$(\cardinalityBin_0,\dots,\cardinalityBin_k)$ of the cardinalities of the sets from \eqref{eq:my-collection}
and then for each index $i\in\{0,\dots,k\}$ we sample $\cardinalityBin_i$ elements of the
interval $[a,b]$. In this case, however, the random variables $\cardinalityBin_0,\dots,\cardinalityBin_k$ are
not independent.

From the above discussion it follows 
that the total variation distance considered
in the statement of this lemma between \ref{TVD:1} and \ref{TVD:2} 
 is equal to the total variation
distance between the random vectors $\cardinalityBin=(\cardinalityBin_0,\dots,\cardinalityBin_k)$ and $\cardinalityPois=(\cardinalityPois_0,\dots,\cardinalityPois_k)$. 
In the following we will bound the latter distance.

The distribution of the random vector 
\[ \cardinalityBin =
\sum_{1\leq j \leq n}  \left( \ind_{\overline{R}^{(j)}=0},  \: \dots , \: \ind_{\overline{R}^{(j)}=k} \right)\] 
is the $n$-fold additive convolution of the discrete probability measure
$\measure$ 
on~$\mathbb{Z}^{k+1}$ which to each basis vector $e_i = [0,\dots,0, 1, 0,\dots, 0] \in \mathbb{Z}^{k+1}$ 
associates the probability $pl$ 
and to the zero vector associates the remaining probability $1-(k+1) pl$.

On the other hand, the distribution of the random vector $\cardinalityPois$
can be alternatively seen as the $n$-fold additive convolution of the product measure
$\mathcal{N}=\operatorname{Pois}\left( \frac{\lambda l}{n} \right) \times \cdots  \times
\operatorname{Pois}\left( \frac{\lambda l}{n} \right)$ on $\mathbb{Z}^{k+1}$. We also consider an auxiliary product measure 
$\mathcal{N}'=\operatorname{Pois}\left( p l\right) \times \cdots  \times
\operatorname{Pois}\left( pl \right)$
on $\mathbb{Z}^{k+1}$.

By \cref{lem:TVD-independent}\ref{item:TVD-convolution} and the triangle inequality
it follows that
\begin{equation}
    \label{eq:m-and-n}
     \delta( \cardinalityBin, \cardinalityPois ) 
  \leq  
 n
\ \delta\left( \measure, \mathcal{N}
\right) \leq 
 n\ \delta\left( \measure, \mathcal{N}'\right) +
 n\ \delta\left(  \mathcal{N}',  \mathcal{N} \right). 
\end{equation}
 For the first summand on the right-hand side, by a direct
 calculation of the positive part of the difference of the two measures and its $\ell^1$ norm, we have
\[  n\ \delta\left( \measure, \mathcal{N}'\right)
= n (k+1) \left[ pl - pl  e^{-pl(k+1)} \right]  \leq 
n (k+1)^2 p^2 l^2  ,
\]
where the inequality follows from an elementary bound on the exponential
function. For the second summand on the right-hand side of \eqref{eq:m-and-n} we apply
\cref{lem:TVD-independent}\ref{item:TVD-convolution} in order to bound the total
variation distance between two Poisson distributions; it follows that
\[   n\ \delta\left(  \mathcal{N}',  \mathcal{N} \right) \leq (k+1)\cdot \left| \lambda l - n p l\right|
\]
which completes the proof.
\end{proof}

\begin{proof}[Proof of \cref{coro:localP}]
    We start with the case when $0<w<1$. We will show a stronger result that for
each $A>0$ the total variation distance between:
    \begin{enumerate}[label=(\roman*)]
        \item \label{item:poisson-p1}
        the collection of $k+1$ sets
\begin{equation}
    \label{eq:truncated}
              \left( \myPset_y \cap [-A, A] \; : \; y\in\{0,\dots,k\} \right), 
\end{equation}
         (cf.~\eqref{eq:myPset}), and 
         \item \label{item:poisson-pp} the
collection of $k+1$ independent Poisson point processes on the interval
$[-A,A]$ with the intensity $\frac{1}{\sqrt{w}}$
		\end{enumerate}
converges to zero, as $n\to\infty$.

\medskip

    As the first step, let us fix $\epsilon>0$. Let $Z_1$ be the number of the
entries of the sequence $\xi_1,\dots,\xi_n$ which are weakly smaller than
$w-\frac{A}{\sqrt{n}}$ and let $Z_2$ be the number of the entries of this
sequence which are strictly smaller than $w+\frac{A}{\sqrt{n}}$;
clearly the probability distribution of $Z_1$ and $Z_2$ is a binomial distribution.
    By Bienaymé--Chebyshev inequality it follows that the constant
\[ B:= A+\frac{1}{\sqrt{\epsilon}} \]
has the property that for each positive integer $n$ 
\newcommand{\nmin}{n_{\operatorname{min}}}
\newcommand{\nmax}{n_{\operatorname{max}}}   
\begin{equation}
    \label{eq:GoodEvent}   
       \mathbb{P} \left( Z_1 <  \nmin \right) \leq \epsilon, \qquad 
              \mathbb{P} \left( Z_2 \geq \nmax \right) \leq \epsilon
    \end{equation}
with
\[ \nmin:=\left\lfloor nw -B \sqrt{n} \right\rfloor,\qquad  
\nmax:=\left\lceil nw +B \sqrt{n} \right\rceil.\]
In the following we assume that $n$ is big enough so that 
\[
1\leq \nmin \leq \nmax \leq n.
\]

\medskip

Without loss of generality we may assume that the entries of the sequence
$\xi_1,\dots,\xi_n$ are not repeated. It follows that this sequence can be
encoded by two pieces of information:
\begin{itemize}
    \item the sequence of order statistics $0<\xi_{(1)}<\cdots<\xi_{(n)}<1$, and
\item the permutation $\pi=(\pi_1,\dots,\pi_n)$ which encodes the order of the
entries, i.e.~$\pi_i<\pi_j$ if and only if $\xi_i<\xi_j$; 
\end{itemize}
these two pieces of information are clearly independent and the permutation $\pi$
is a uniformly random element of the symmetric group.

Since RSK algorithm is sensitive only to the relative order of the entries and
not to their exact values, the insertion tableaux $P(\xi_1,\dots,\xi_n)$ can be
obtained from the insertion tableau $P(\pi_1,\dots,\pi_n)$ by replacing each
entry by the corresponding order statistic. It follows that the entries of a
given row $y \in \N_0$ of the insertion tableau can be alternatively described as follows:
\begin{multline*} 
\left\{  P^{(n)}_{x,y} \ : \ 0\leq x < \lambda^{(n)}_y \right\} = \\
   = \left\{ \xi_{(i)} :  \text{$i$ is in the row $y$ of the tableau $P(\pi_1,\dots,\pi_n)$} \right\} 
.    \end{multline*}

It follows in particular that the intersection $\myPset_y \cap [-A, A]$ is
defined in terms of (a certain subset of) the set of these order statistics $\xi_{(i)}$ which belong to the interval 
\[ I:= [a,b]\]
with
\[ a:=  w - \frac{A}{\sqrt{n}}, \qquad \qquad b:= w +  \frac{A}{\sqrt{n}} . \]
If neither of the two random events appearing in
\eqref{eq:GoodEvent} holds true then this set of order statistics fulfills
\[ 
I \cap \left\{ \xi_{(1)}, \dots, \xi_{(n)} \right\} 
\subseteq \left\{ \xi_{(\nmin)}, \dots, \xi_{(\nmax)} \right\}; 
\]
under this condition it follows that in order to find the number of the row of
$P^{(n)}$ which contains a given order statistic $\xi_{(i)}\in I$ it is enough to
know the number of the row of the tableau $P(\pi_1,\dots,\pi_n)$ which contains a
given number $j$, over all choices of $j\in\{ \nmin,\dots,\nmax\}$.
Here and in the following by \emph{the number of the row} 
we will understand the element of
the fixed finite set $\mathcal{N}=\{0,\dots,k,\infty\}$ with the convention that
the element $\infty$ corresponds to all rows above the bottom $k+1$ rows.

The insertion tableau $P(\pi)=Q\left( \pi^{-1}\right)$ is equal to the recording
tableau of the inverse permutation; as a consequence the probability distribution
of $P(\pi)$ is given by the Plancherel measure, can be interpreted as (a part of)
the Plancherel growth process and thus \cref{thm:independent-explicit} is
applicable to this tableau. It follows that the probability distribution of the
vector formed by the so understood \emph{numbers of the rows} of the boxes
$\nmin,\dots,\nmax$ can be approximated (up to an error $o(1)$ with respect to
the total variation distance) by a sequence of independent random variables with
the probability distribution given by \eqref{eq:funny-probability-dist} for
$p= \frac{1}{\sqrt{\nmin}}$. Here and in the following we assume that $n$ is big enough so that $(k+1) p \leq 1$.

The above two paragraphs show that the total variation distance between the
collection 
\begin{equation}
    \label{eq:truncated-2}
    \left( [a,b]\cap \left\{ P_{x,y}^{(n)} : x\geq 0 \right\}   : y\in\{0,\dots,k\} \right), 
\end{equation}
of truncated entries of the bottom rows
and the collection \eqref{eq:my-collection} is
bounded from above by $2\epsilon+o(1)$. On the other hand, \cref{lem:TVD-Poisson}
applied to $\lambda=\sqrt{\frac{n}{w}}$
shows that the total variation distance between \eqref{eq:my-collection} and the
collection of $k+1$ independent Poisson point processes \ref{TVD:2} with the intensity $\lambda$
on the interval $I$ converges to zero. We combine these two bounds by the triangle inequality;
as a result the total variation distance between \eqref{eq:truncated-2} and \ref{TVD:2} is bounded from above by $2\epsilon+o(1)$.

We consider the affine transformation $x\mapsto \sqrt{n}\ (x-w)$ which maps the
interval $[a,b]$ to $[-A,A]$. This affine transformation also maps the random
collection \eqref{eq:truncated-2} to \eqref{eq:truncated} from \ref{item:poisson-p1};
it also maps the collection of
Poisson processes  \ref{TVD:2} from \cref{lem:TVD-Poisson} with the intensity
$\lambda$ to the collection \ref{item:poisson-pp} of Poisson processes with the
intensity $\lambda \frac{1}{\sqrt{n}}=\frac{1}{\sqrt{w}}$. The application of this
affine transformation preserves the total variation distance between the random
variables, so the inequality from the previous paragraph completes the proof.

\bigskip

In the case when $w=1$ the above proof can be easily adjusted by changing the definitions of
$\nmax:=n$ and $b:=1$.
\end{proof}

\section{Acknowledgments}

Research supported by Narodowe Centrum Nauki, grant number \linebreak 2017/26/A/ST1/00189.
Mikołaj Marciniak was additionally supported by 
Narodowe Centrum Badań i Rozwoju, grant number POWR.03.05.00-00-Z302/17-00.

We thank Iskander Azangulov, Maciej Dołęga, Vadim Gorin, Piet Groeneboom, Adam
Jakubowski, Grigory Ovechkin, Timo Sepp\"{a}l\"{a}inen, and Anatoly Vershik for
discussions and bibliographic suggestions.

\biblio

\end{document}